\documentclass[11pt,a4paper]{article}

\usepackage[T1]{fontenc}
\usepackage{epsfig}
\usepackage{graphicx}
\usepackage{amssymb}
\usepackage{amsmath}
\usepackage[utf8]{inputenc}
\usepackage{amsfonts}
\usepackage{amsthm}
\usepackage{stmaryrd}
\usepackage{psfrag,romanbar}

\usepackage{tikz}
\usepackage{tikz-qtree}
\usetikzlibrary{fit,calc,positioning,decorations.pathreplacing,matrix,trees,patterns}
\usepackage{xcolor}

\usepackage{geometry}
\usepackage{hyperref}
\usepackage[numbers]{natbib}

\makeatletter

\@addtoreset{equation}{section}
\makeatother

\theoremstyle{plain}
\newtheorem{theorem}{Theorem}[section]
\newtheorem{corollary}[theorem]{Corollary}
\newtheorem{lemma}[theorem]{Lemma}

\newtheorem{proposition}[theorem]{Proposition}
\newtheorem{property}[theorem]{Property}
\newtheorem{conj}[theorem]{Conjecture}

\theoremstyle{definition}
\newtheorem{definition}[theorem]{Definition}

\theoremstyle{remark}

\newtheorem{remark}[theorem]{Remark}

\DeclareMathOperator*{\argmin}{argmin}

\newcommand{\RST}{\text{\textbf{RST}}}

\newcommand{\CFD}{\text{CFD}}
\newcommand{\MBD}{\text{MBD}}

\newcommand{\Stab}{\text{Stab}}
\newcommand{\Cone}{\text{Cone}}

\newcommand{\Rad}{\text{Rad}}

\newcommand{\Vol}{\mbox{Vol}}

\makeindex

\newcommand{\NN}{{\mathcal{N}}}

\newcommand{\eps}{\varepsilon}

\newcommand{\E}{{\mathbb E}}

\renewcommand{\P}{{\mathbb P}}

\newcommand{\R}{{\mathbb R}}
\renewcommand{\H}{{\mathbb H}}
\newcommand{\ind}{{\textbf{1}}}

\renewcommand{\Stab}{\mbox{Stab}}

\title{Thick trace at infinity\\ for the Hyperbolic Radial Spanning Tree}

\author{David Coupier\footnote{IMT Nord Europe, Institut Mines T\'el\'ecom, Univ. Lille, david.coupier@imt-nord-europe.fr}, \qquad Lucas Flammant\footnote{LAMAV, Univ. Polytechnique Hauts-de-France, CNRS}, \qquad Viet Chi Tran\footnote{LAMA, Univ Gustave Eiffel, Univ Paris Est Creteil, CNRS; IRL 3457, CRM-CNRS, Université de Montréal, Canada. chi.tran@univ-eiffel.fr}}

\date{\today}

\begin{document}

\maketitle

\begin{abstract}
Since the works of Howard \& Newman (2001), it is known that in \textit{straight} radial rooted trees, with probability $1$, infinite paths all have an asymptotic direction and each asymptotic direction is reached by (at least) an infinite path. Moreover, there exists a set of 'exceptionnal' directions reached by (at least) two infinite paths which is random, dense and only countable in dimension $2$. Howard \& Newman's method says nothing about (random) directions reached by more than two infinite paths and, in particular, if such 'very exceptionnal' directions exist in dimension $2$. In this paper, we prove that the answer is no for the hyperbolic Radial Spanning Tree (RST): in dimension $2$, this tree does not contain $3$ infinite paths with the same (random) asymptotic direction with probability one. Turned in another way, this means that there is no infinite but thin subtree in the hyperbolic RST, i.e. whose infinite paths would all have the same asymptotic direction. We actually prove a stronger result in dimension $d+1$, $d\geq 1$, stating that any infinite subtree of the hyperbolic RST a.s. generates a thick trace at infinity, i.e. the set of asymptotic directions reached by its infinite paths has a positive measure.
\end{abstract}

\bigbreak
\bigbreak
\textbf{Acknowledgments.} This work has been supported by the RT GeoSto 3477, the Labex Bézout (ANR-10-LABX-58), the ANR PPPP (ANR-16-CE40-0016) and the ANR GrHyDy (ANR-20-CE40-0002). V.C.T. thanks the CRM Montr\'eal (IRL CNRS 3457) where part of this work was completed.

\section{Introduction}

Unlike combinatorial random graphs whose Erd\"os-R\'enyi model is certainly the most famous specimen (see \cite{ErdosRenyi,vanderhofstad}), the structure of a \textit{geometric} random graph depends on the locations of its vertices (embedded in a metric space) and on some geometrical rules. This feature makes geometric random graphs suitable to model real phenomena (in biology, material science, image analysis or telecommunication networks) and this is the reason why they have been intensively studied in the last decades. See the book of Penrose \cite{penrose2003random} for a general reference on the topic.

For geometric random graphs, macroscopic properties trigger challenging questions, such as studying the number of the topological ends of these structures. Alexander solves this question in \cite{Alexander} for Minimal Spanning Forests on infinite graphs. In the case of Euclidean trees, i.e. trees whose vertices are points of $\R^{d+1}$ with $d\geq 1$, a fundamental step has been taken by Howard \& Newman for \textit{straight} rooted trees, i.e. whose subtrees are all the thinner as their roots are far away from that of the entire tree \cite[Section 2.3]{HowardNewman}. They develop an efficient method \cite[Proposition 2.8]{HowardNewman} ensuring that any straight tree $T$ satisfies the two following properties almost surely (a.s.):

\medskip

\noindent
(A) Every infinite path $(z_n)_{n\geq 0}$ of $T$-- a sequence of different vertices $(z_n)_{n\geq 0}$ such that $\{z_n,z_{n+1}\}$ is an edge of the tree for every $n\geq 0$ --admits an asymptotic direction $u$ in the unit sphere $\mathbb{S}^{d}$ of $\R^{d+1}$, i.e.
\begin{equation}
\label{AsymptoticDirection}
\lim_{n \rightarrow +\infty} \frac{z_n}{|z_n|} = u ~.
\end{equation}
\noindent
(B) Every direction $u \in \mathbb{S}^{d}$ is asymptotically reached by (at least) one infinite path of the tree $T$. 

\medskip

\noindent
We will also say that the direction $u$ satisfying (\ref{AsymptoticDirection}) is \textit{asymptotically reached} or \textit{targeted} by the path $(z_n)_{n\geq 0}$. In the whole paper, the dimension of the ambiant space is denoted by $d+1$.

Howard and Newman applied their method to the case of first-passage percolation trees built on a Poisson Point Process (PPP). They also proved that, for any deterministic $u \in \mathbb{S}^{d}$, the tree $T$ a.s. contains exactly one infinite path with $u$ as asymptotic direction. However there exists a.s. a set of random directions, thought as 'exceptional', targeted by at least two infinite paths which is dense in $\mathbb{S}^{d}$, and only countable in dimension $2$. See Fig. \ref{Fig:simulationRST} and Remark \ref{Rmk:TwoPaths} for details about these 'exceptional' directions.

Howard \& Newman's method then motivated many works on geometric random trees (in particular in dimension $2$). Various authors stated the straightness of their favorite tree so as to obtain Properties (A) and (B) for its infinite paths. Here are some examples. The Directed Last-Passage Percolation (LPP) Tree is obtained by a last-passage procedure from i.i.d. weights associated to the vertices of the grid $\mathbb{N}^{d+1}$. In the case of exponential weights and $d+1=2$, Ferrari \& Pimentel \cite{FP} established the straightness of the directed LPP tree. From asymptotic directions of infinite paths of the LPP tree, they deduced the existence of an asymptotic direction for a competition interface. Other works focused on geometric trees directed towards a distinguished point $0$, with edges defined by local geometric rules. A first example is the Radial Spanning Tree (RST) introduced by Baccelli \& Bordenave in the Euclidean plane to model communication networks \cite{BB07}. Each vertex of this tree has an outgoing edge towards the nearest vertex among those closer to $0$. The hyperbolic RST studied in this paper is an extension of their RST. Using a directed forest (namely the Directed Spanning Forest) approximating locally, in distribution and far from the root the RST, they proved the straightness of the RST. A second example is the Navigation Tree defined by Bonichon \& Marckert in \cite{BM}. Each vertex has also outdegree one and the corresponding edge links it to the closet vertex in a given cone directed towards $0$. To describe the asymptotic geometry of the navigations paths, the authors stated the straightness of the Navigation Tree.\\

Since in dimension $2$, only a countable number of random directions are targeted by at least two infinite paths, it is natural to ask whether some of them are `very exceptional', i.e. targeted by more than two infinite paths. Assuming the non-crossing path property (satisfied by all the previously mentioned trees), the existence of such a 'very exceptional' direction would mean that of an infinite subtree-- containing the middle infinite path --which would be very thin because trapped between both extreme infinite paths having the same asymptotic direction. The existence of such a subtree is suspicious since it is involved in a spatial competition that it neither wins (it fails to occupy a macroscopic part of the space) nor loses (it is unbounded). This heuristic sustains the following refinement of the Howard \& Newman's results:

\begin{conj}
\label{conj1}
For most of straight bi-dimensional geometric random trees having the non-crossing path property which are studied in the literature (including those cited above), there is no (random) direction targeted with more than two infinite paths with probability $1$.
\end{conj}

To our knowledge, this conjecture has been established in only one case: for the Directed LPP Tree with exponential weights in $\mathbb{N}^2$ by Coupier \cite{CoupierLPP}. His proof is based on a surprising coupling-- exhibited for the first time by Rost \cite{Rost} --between the LPP model and the Totally Asymmetric Simple Exclusion Process (abbreviated in the literature to TASEP) and on results on some special particles (called second class particles) of this particle system \cite{AAV}.

In this paper, we present a second example satisfying this conjecture, namely the hyperbolic Radial Spanning Tree introduced in \cite{flammant-RST}. See Fig. \ref{Fig:simulationRST}. Let us emphasize that the nature of this new example is completely different from the first one since it lies in a continuum hyperbolic space instead of $\mathbb{N}^2$.

\begin{figure}[!h]
\centering
\includegraphics[width=9cm,height=9cm]{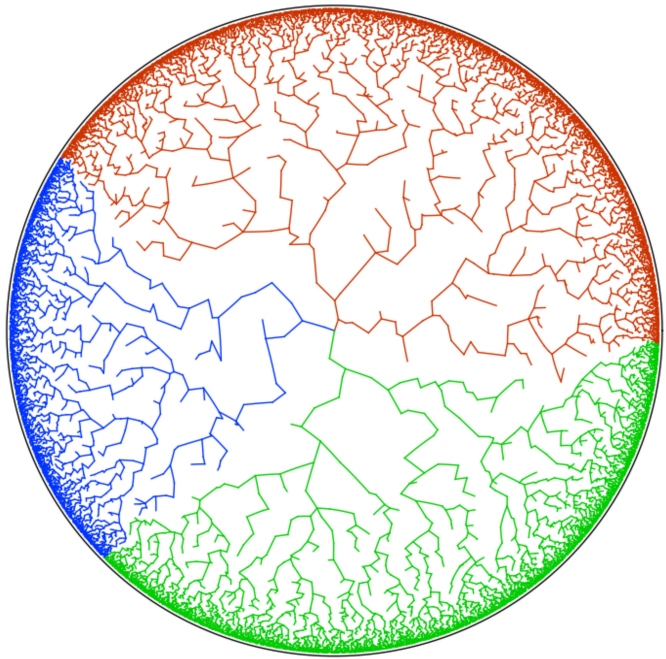}
\caption{Simulation of the two-dimensional hyperbolic RST, with $\lambda=30$, in the Poincaré disc model. The edges are represented by geodesics. The different connected components of the RST (apart from the root) are represented with different colors. This coloring allows to distinguish some random directions reached by two infinite paths, namely the three directions separating the three colored traces at infinity: see Remark \ref{Rmk:TwoPaths} for details.}
\label{Fig:simulationRST}
\end{figure}

The construction of the hyperbolic RST is the same as for the bi-dimensional Euclidean RST of Baccelli \& Bordenave \cite{BB07}, whatever the dimension $d+1\geq 2$. The set of vertices is given by a homogeneous Poisson Point Process (PPP) $\mathcal{N}$ of intensity $\lambda>0$. The RST rooted at the origin $0$ is the graph obtained by connecting each point $z \in \mathcal{N}$ to its parent $A(z)$, defined as the closest point to $z$ among all points $z' \in \mathcal{N} \cup \{0\}$ that are closer to the origin than $z$. This procedure defines a random radial tree rooted at $0$ which is \textit{straight} in $\mathbb{H}^{d+1}$; see \cite[Proposition 2.7]{flammant-RST} or Property \ref{Prop:FiniteDegree} below. This key property is available here thanks to the fact that the hyperbolic metric guarantees that angular deviations of RST paths decay exponentially fast with the distance to the origin. This feature entails in fact much more than the straightness of the hyperbolic RST. It also implies that there is a positive proportion of RST edges at level $r$ giving rise to infinite paths (Proposition \ref{Propo:DensitySigma>0} below). This later statement does not occur in the Euclidean case, where the probability that a given edge at level $r$ belongs to an infinite path tends to 0 as $r \to \infty$ \cite{baccellicoupiertran}.\\

Our main result (Theorem \ref{Thm:positivevolume}) holds in any dimension. Heuristically, we will show that the set of directions $u\in \mathbb{S}^d$ that are asymptotically reached by infinite paths stemming from an arbitrary Poisson point $z$, called the \textit{trace of $z$ at infinity}, is either empty or has positive measure. In the first case, the subtree rooted at $z$ is finite while in the second case, it is infinite and generates a thick trace at infinity. In dimension $2$, thanks to planarity and the non-crossing path property, it is easily deduced from Theorem \ref{Thm:positivevolume} that a.s. there is no direction reached by more than two infinite paths, i.e. Conjecture \ref{conj1}. See Corollary \ref{Cor:positivevolume}, Item $(iii)$.

The proof of Theorem \ref{Thm:positivevolume} can be summarized by the geometric construction depicted in Fig. \ref{fig:ConstructionU} and mainly relies on the fact that the probability with which a given Poisson point $z$ at level $r$ generates a thick trace at infinity (plus extra conditions as in particular a stabilization criteria) is bounded away from $0$ \textit{uniformly on} $r$ (Lemma \ref{Lem:unifF0}). This technical result is specific to the hyperbolic metric: this explains why we are able to prove Conjecture \ref{conj1} for the hyperbolic RST and not for its Euclidean counterpart (although the Euclidean RST contains in proportion fewer infinite paths than the hyperbolic RST). Lemma \ref{Lem:unifF0} also uses a control of path fluctuations developed in \cite{flammant-RST} (recalled Lemma \ref{lem:MBD}).\\

In Section \ref{Section:hyp_geom}, we recall shortly some results on the hyperbolic geometry that will be useful in the paper. In Section \ref{Section:MainResults}, we enounce our main result: Theorem \ref{Thm:positivevolume} for the general dimension $d+1$, $d\geq 1$, and its Corollary \ref{Cor:positivevolume} in dimension $2$, that answers the Conjecture \ref{conj1} for the hyperbolic Radial Spanning Tree, providing the first example in continuum space where this question can be answered. Section \ref{Section:PositiveVolume} presents the proofs of Theorem \ref{Thm:positivevolume} and of the technical Lemma \ref{Lem:unifF0} introduced in the previous paragraph. Section \ref{Section:UnifF0} is devoted to the proof of this later result. Finally, in Section \ref{Section:Densite>0}, we prove that there is a positive proportion of points of the tree at level $r$ that generates thick tracks at infinity.

\section{Hyperbolic geometry and notations}\label{Section:hyp_geom}

\noindent
\textbf{Generalities on hyperbolic geometry.} We refer to \cite{cannon,chavel,paupert,Ratcliffe} for a complete introduction to hyperbolic geometry. For $d \in \mathbb{N}^*=\{1,2,\dots\}$, the $(d+1)$-dimensional hyperbolic space denoted by $\mathbb{H}^{d+1}$ is a $(d+1)$-dimensional Riemannian manifold of constant negative curvature $-1$ that can be defined by several isometric models. One of them is the open-ball model $\mathfrak{D}$ (or Poincar\'e disc model in dimension $2$) consisting in the unit open-ball
\[
\mathfrak{D} := \{(x_1,\ldots,x_{d+1}) \in \mathbb{R}^{d+1} : \, x_1^2+\ldots+x_{d+1}^2 < 1\}
\]
endowed with the metric:
\[
ds_\mathfrak{D}^2 := 4\frac{dx_1^2+\ldots+dx_{n+1}^2}{(1-x_1^2-\ldots-x_{d+1}^2)^2} ~.
\]
The volume measure on $(\mathfrak{D},ds_\mathfrak{D}^2)$ is then given by $2^{d+1} dx_1\ldots dx_{d+1}/(1-x_1^2-\ldots-x_{d+1}^2)^{d+1}$.\\

A convenient way to represent points in the open-ball model is to use polar coordinates (w.r.t. the origin $0$). Any $z \in \mathfrak{D}$ can be written as $z=(r;u)$ where $r=d(z,0)$ is its distance to $0$ and $u \in \mathbb{S}^d$ is its direction. Let $\sigma$ be the spherical probability measure on $\mathbb{S}^d$ (in particular $\sigma(\mathbb{S}^d)=1$). In polar coordinates, the \textit{rescaled} volume measure $\Vol$, corresponding to the rescaled probability measure $\sigma$, becomes
\begin{eqnarray}
\label{E:dvolpolar}
d\Vol(r;u) = \sinh(r)^d \,dr \,d\sigma(u) ~.
\end{eqnarray}
The choice of rescaling of the volume measure is adopted to avoid writing the constant $2\pi^{(d+1)/2}/\Gamma((d+1)/2)$ that would appear a lot in the paper otherwise (this constant is the non-rescaled measure of $\mathbb{S}^d$, see  \cite[(III.3.10) p.125]{chavel}).\\

The hyperbolic space $\mathbb{H}^{d+1}$ is also naturally equipped with a set of \textit{points at infinity} denoted by $\partial \mathbb{H}^{d+1}$. In the open-ball model, this set is identified with the unit sphere $\mathbb{S}^d$. The distances are distorted in comparison with the Euclidean distance and become `smaller' when we approach the boundary $\partial \mathfrak{D}=\mathbb{S}^d$, which is at infinite hyperbolic distance from the center $0$. We also set $\overline{\mathbb{H}^{d+1}} := \mathbb{H}^{d+1} \cup \partial \mathbb{H}^{d+1}$ endowed with the topology given by the closed ball. 

In $(\mathfrak{D},ds_\mathfrak{D}^2)$, geodesics are of two types: either  diameters of $\mathfrak{D}$ or arcs perpendicular to the boundary $\partial \mathfrak{D}$. Also, this model is conformal in the sense that the hyperbolic angle between two geodesics is equal to the Euclidean angle between them, in the open ball representation. Another important fact about hyperbolic geometry is that all points and all directions play the same role: $\mathbb{H}^{d+1}$ is homogeneous and isotropic.\\

\noindent
\textbf{Notations.} We denote by $d(\cdot,\cdot)$ the hyperbolic distance in $\mathbb{H}^{d+1}$. For $z,z' \in \overline{\mathbb{H}^{d+1}}$, let $[z,z']$ be the geodesic between $z$ and $z'$. We will denote by $(z,z')$ the geodesic without the extremities $z$ and $z'$.

For $z \in \mathbb{H}^{d+1}$ and $r>0$, we respectively denote by $B(z,r) := \{z' \in \mathbb{H}^{d+1} : d(z,z')<r\}$ and $S(z,r) := \{z' \in \mathbb{H}^{d+1} : d(z,z')=r\}$ the hyperbolic open ball and hyperbolic sphere centered at $z$ with radius $r>0$. We will write $B(r):=B(0,r)$ and $S(r):=S(0,r)$ for short. Writting $\Vol(B(r)) = \int_{0}^{r} \sinh^d(\rho) d\rho$ (see e.g. \cite[p.79]{Ratcliffe}), it is easy to prove that there exists $c = c(d) \in (0,1)$ such that for any $r\geq 1$,
\begin{equation}
\label{InegVol}
c e^{dr} \leq \Vol(B(r)) \leq \nu e^{dr} ~.
\end{equation}
Let us also denote by $C(z,r,R) := \{z' \in \mathbb{H}^{d+1} : r<d(z,z')<R\}$ the annulus with radii $r<R$ and $C(r,R):=C(0,r,R)$.

For $z,z',z'' \in \overline{\mathbb{H}^{d+1}}$, $\widehat{zz'z''}$ is the measure of the corresponding (non-oriented) hyperbolic angle. For any $z \in \overline{\mathbb{H}^{d+1}}$ and $\theta>0$, $\Cone(z,\theta):=\{z' \in \mathbb{H}^{d+1} : \widehat{z 0 z'} \le \theta\}$ is defined as the cone of apex $0$, axis $z$ and aperture $\theta$. In addition, for $z \in \mathbb{H}^{d+1}$ and $r>0$, we will use the spherical cap
\begin{eqnarray}
\label{E:defBSR}
B_{S(r)}(z,\theta) := \Cone(z,\theta) \cap S(r) ~.
\end{eqnarray}
For its surface, note that there exists a constant $C=C(d)>0$ such that:
\begin{equation}
\label{area_cap}
\sigma\big(B_{S(r)}(z,\theta)\big) \leq C \theta^{d} ~.
\end{equation}

\section{Main results}
\label{Section:MainResults}

In the sequel, we pick some arbitrary origin point $0 \in \mathbb{H}^{d+1}$ (thought as the center of $\mathfrak{D}$ in the open-ball representation) to be the root of the hyperbolic RST that we now define.\\

\noindent
\textbf{The hyperbolic RST.} Let $\mathcal{N}$ be a homogeneous PPP of intensity $\lambda>0$ in $\mathbb{H}^{d+1}$. The definition of the hyperbolic RST is similar to the Euclidean case. This is a directed graph whose vertex set is $\mathcal{N} \cup \{0\}$ and in which each vertex $z \in \mathcal{N}$ is connected to the closest Poisson point among $(\NN\cup \{0\}) \cap B(r)$, with $r := d(0,z)$, for the hyperbolic distance.

\begin{definition}[Radial Spanning Tree in $\mathbb{H}^{d+1}$]
\label{Def:rsthyp}
The \emph{ancestor} of $z \in \mathcal{N}$ is defined as
\begin{equation}\label{def:A(z)}
A(z) := \argmin_{z' \in (\mathcal{N}\cup \{0\}) \cap B(d(0,z))} d(z',z) ~.
\end{equation}
The \emph{Radial Spanning Tree} (RST) in $\mathbb{H}^{d+1}$ is the directed graph $(V,\vec{E})$ where $V := \mathcal{N} \cup \{0\}$ and $\vec{E} := \{(z,A(z)) : z \in \mathcal{N}\}$.
\end{definition}

Since $\mathcal{N} \cup \{0\}$ contains no isosceles triangles with probability $1$, the ancestor $A(z)$ of any Poisson point $z$ is a.s. well-defined. In other words, any Poisson point admits only one outgoing edge, but possibly several ingoing edges. In any case, these ingoing edges are finitely many:

\begin{property}[Proposition 2.2 of \cite{flammant-RST}]
\label{Prop:FiniteDegree}
A.s. the RST is a tree rooted at $0$ where all vertices have finite degrees. In the bi-dimensional case ($d=1$), the union of geodesics $[z,A(z)]$, $z\in \NN$, is planar, i.e. whatever $z\not= z' \in \NN$, the geodesics $[z,A(z)]$ and $[z',A(z')]$ may only overlap on their endpoints.
\end{property}

For $z \in \mathcal{N}$ and $r > 0$, we also define 
\begin{equation}
\label{def:Bplus}
B^+(z,r) := B(z,r) \cap B(d(0,z)) \quad \mbox{ and } \quad B^+(z) := B^+(z,d(z,A(z))) ~.
\end{equation}
By construction of the ancestor, the random set $B^+(z)$ avoids the PPP $\mathcal{N}$. This fact is responsible for many difficulties when studying the RST. Indeed, when restarting from $A(z)=(r';u')$ and constructing the path forward (towards $0$), with probability $1$, $B(r') \cap B^+(z)$ is non-empty. This means that the geometric information used to determine $A(z)$ is still involved for next steps of the process, generating statistical dependencies.\\

A \textit{path} of the RST is a sequence (finite or not) of different vertices $(z_0,z_1,z_2\ldots)$ in $\NN\cup\{0\}$ such that $A(z_{n+1})=z_n$ for any $n\geq 0$. By \eqref{def:A(z)}, we have that for all $z\in \NN$, $d(0,A(z))<d(0,z)$. Applying this to $z=z_n$, $n\geq 1$, we obtain that $d(0,z_{n})<d(0,z_{n+1})$. The path $(z_0,z_1,z_2\ldots)$ can thus be viewed as an exploration of the RST starting at $z_0$ and moving away from $0$. We will say that the forward direction is towards $0$ and the backward direction is towards infinity. Every path can be started at $z_0=0$, but this is not an obligation.

The \textit{set of descendants} $\mathcal{D}(z)$ of a given vertex $z$ is made up of all vertices that can be reached by a backward finite path starting at $z$ (including $z$ itself by convention). Because a vertex can be the ancestor of no other vertex, all sets of descendants are not infinite.

As mentioned in the Introduction, a very important feature of the hyperbolic RST is its \textit{straightness}:

\begin{property}[Proposition 2.7 of \cite{flammant-RST}]
\label{Prop:straightness}
A.s. for any $\varepsilon>0$ there exists some $r_0>0$ such that, for any radius $r \ge r_0$ and for any vertex $z$ of the hyperbolic RST with $d(0,z) \ge r$, the set of descendants $\mathcal{D}(z)$ is contained in a cone of apex $0$ and aperture $e^{-(1-\varepsilon) r}$, i.e. for any $z',z'' \in \mathcal{D}(z)$, $\widehat{z'0z''} \le e^{-(1-\varepsilon) r}$.
\end{property}

\noindent
\textbf{Trace on the boundary $\partial \mathbb{H}^{d+1}$.} Let us consider a point $z\in \NN$ and its set of descendants $\mathcal{D}(z)$. When it is infinite, it contains (at least) an infinite path $(z_n)_{n\geq 0}$ by the finite degree property (Property \ref{Prop:FiniteDegree}). Theorem 1.1 of \cite{flammant-RST} then asserts that the path $(z_n)_{n\geq 0}$ admits an asymptotic direction $u \in \partial \mathbb{H}^{d+1}$. Thus, let us define by $\mathcal{D}^\infty(z)$ the set of asymptotic directions in $\partial \mathbb{H}^{d+1}$ that can be reached by the infinite paths of $\mathcal{D}(z)$. Roughly speaking, $\mathcal{D}^\infty(z)$ is the trace left by the set of descendants $\mathcal{D}(z)$ on the boundary $\partial \mathbb{H}^{d+1}$ (see the traces at infinity left by the children of $0$ in Fig. \ref{Fig:simulationRST}). We just have proved that:

\begin{property}
A.s. for any vertex $z\in \NN$,
\begin{equation}
\label{DetDinfty}
\# \mathcal{D}(z) = +\infty \, \Longleftrightarrow \, \mathcal{D}^\infty(z) \not= \emptyset ~.
\end{equation}
\end{property}

Our main result (Theorem \ref{Thm:positivevolume}) establishes a stronger statement: infinitely many descendants in $\mathcal{D}(z)$ actually implies that $\mathcal{D}(z)$ leaves a trace with positive volume at infinity, i.e. $\sigma(\mathcal{D}^{\infty}(z)) > 0$. In this case, we will say that $z$ generates a thick trace at infinity.

\begin{theorem}
\label{Thm:positivevolume}
Consider the hyperbolic RST in $\mathbb{H}^{d+1}$, for any dimension $d\geq 1$. A.s. for any vertex $z\in\NN$, either $z$ admits finitely many descendants or $\sigma(\mathcal{D}^{\infty}(z)) > 0$.
\end{theorem}

\noindent
\textbf{Consequence in $\mathbb{H}^{2}$.} Theorem \ref{Thm:positivevolume} holds in any dimension. In dimension $2$ (i.e. with $d=1$), combined to the non-crossing path property (Property \ref{Prop:FiniteDegree}), it leads to Conjecture \ref{conj1}: a.s. the RST does not contain three infinite paths with the same (random) asymptotic direction. This statement-- Item $(iii)$ of Corollary \ref{Cor:positivevolume} --completes the description of infinite paths and their asymptotic directions of the hyperbolic RST started in \cite{flammant-RST}. The first three items below are given by \cite[Theorem 1.1]{flammant-RST}. Their proofs are based on the strategy developed by Howard and Newman in \cite{HowardNewman} and on the \textit{straightness} of the hyperbolic RST. 

\begin{corollary}
\label{Cor:positivevolume}
The following properties concern the hyperbolic RST in $\mathbb{H}^2$.
\begin{itemize}
\item[$(o)$] A.s. any infinite path admits an asymptotic direction and, for any  $u \in \partial \mathbb{H}^2$, the RST contains an infinite path with asymptotic direction $u$.
\item[$(i)$] For any (deterministic) $u \in \partial \mathbb{H}^2$, the RST a.s. contains a unique infinite path with asymptotic direction $u$.
\item[$(ii)$] A.s. the subset of $\partial \mathbb{H}^2$ of asymptotic directions reached by at least two infinite paths is dense and countable in $\partial \mathbb{H}^2$.
\item[$(iii)$] A.s. no (random) asymptotic direction of $\partial \mathbb{H}^2$ is reached by more than two infinite paths.
\end{itemize}
\end{corollary}

Item $(o)$ says that any $u$ in $\partial \mathbb{H}^2$ is reached by at least one infinite path and exactly one when $u$ is deterministic (Item $(i)$). There exist by Item $(ii)$ asymptotic directions which are reached by several infinite paths but these directions are random and few (only countable). Finally, Theorem \ref{Thm:positivevolume} specifies that there is no random asymptotic direction reached by more than two infinite paths.

\begin{proof}[Proof of Corollary \ref{Cor:positivevolume}, Item $(iii)$]
Let us assume that the hyperbolic RST contains three different infinite paths, say $(x_n)_{n\geq 0}$, $(y_n)_{n\geq 0}$ and $(z_n)_{n\geq 0}$ having the same asymptotic direction in $\partial \mathbb{H}^2$. Without loss of generality, we can also assume that these three paths have no vertices in common. By the non-crossing path property and planarity, one of these three infinite paths, say $(y_n)_{n\geq 0}$, is trapped between the two other paths which by hypothesis have the same asymptotic direction. This forces $\sigma(\mathcal{D}^{\infty}(y_0)) = 0$ while the set of descendants $\mathcal{D}(y_0)$ is infinite. This occurs with null probability thanks to Theorem \ref{Thm:positivevolume}.
\end{proof}

\begin{remark}
\label{Rmk:TwoPaths}
Let us show that the asymptotic direction separating the green and red traces at infinity on Fig. \ref{Fig:simulationRST}, is reached by two infinite paths (a green one and a red one). In the green (resp. red) infinite subtree, pick the leftmost (resp. rightmost) infinite path in the trigonometric sense. Such construction makes sense in dimension $2$ thanks to the non-crossing path property. These two paths $\pi_{\textrm{green}}$ and $\pi_{\textrm{red}}$ admit asymptotic directions, say $u_{\textrm{green}}$ and $u_{\textrm{red}}$ in $\partial \mathbb{H}^2$ by Corollary \ref{Cor:positivevolume}, Item $(o)$. If $u_{\textrm{green}}$ and $u_{\textrm{red}}$ were different, any asymptotic direction $u'$ located (strictly) between them shoud be reached by an infinite path $\pi$ thanks to Corollary \ref{Cor:positivevolume}, Item $(o)$. By construction of $\pi_{\textrm{green}}$ and $\pi_{\textrm{red}}$, the path $\pi$ could not be neither green or red, leading to a contradiction.\\
\end{remark}

\noindent
\textbf{Positive density with $\sigma(\mathcal{D}^\infty(\cdot)) > 0$.} Although the previous results concern the combinatorial structure of the RST, it will be useful in the proofs to represent the graph RST as a subset of $\mathbb{H}^{d+1}$, denoted by $\RST$, in which each edge $(z,A(z))$ is represented by the \textit{arc} $|\![z,A(z)]\!|$ (defined in Appendix \ref{Appendix:arcs}) and not by the (hyperbolic) geodesic $[z,A(z)]$:
\begin{equation}
\RST := \bigcup_{z\in \NN} |\![z,A(z)]\!| ~.
\end{equation}
This choice (somewhat unnatural) is motivated by the fact that the (hyperbolic) distance to the origin $0$ is monotonous along the arc $|\![z,A(z)]\!|$ which fails for the geodesic $[z,A(z)]$ (or for the Euclidean segment between $z$ and its ancestor $A(z)$). In particular each arc $|\![z,A(z)]\!|$ crosses any given sphere at most once. Thus, we define the RST \textit{at level} $r>0$ as the following random set:
\begin{equation}
\mathcal{L}_r := \RST \cap S(r) ~.
\end{equation}
Elements of $\mathcal{L}_r$ are a.s. not Poisson points on $S(r)$ with probability $1$. However we extend to elements of $\mathcal{L}_r$ the notations $\mathcal{D}(\cdot)$ and $\mathcal{D}^\infty(\cdot)$ as follows. For $z\in \mathcal{L}_r$, we denote by $z_{\downarrow}$ the Poisson point whose arc $|\![z_{\downarrow}, A(z_{\downarrow})]\!|$ crosses $S(r)$ at $z$: with a slight abuse of notations, $z_{\downarrow}$ can be defined without ambiguity (see Appendix). Then we set $\mathcal{D}(z)=\mathcal{D}(z_{\downarrow})$ and $\mathcal{D}^\infty(z)=\mathcal{D}^\infty(z_{\downarrow})$.

This section ends with a density result. Proposition \ref{Propo:DensitySigma>0} heuristically says that in expectation a macroscopic proportion of elements of $\mathcal{L}_r$ generates a thick trace at infinity:

\begin{proposition}
\label{Propo:DensitySigma>0}
There exists $c=c(d)>0$ such that for any $r>0$,
\[
\E \big[ \# \{ z \in \mathcal{L}_{r} : \, \sigma(\mathcal{D}^{\infty}(z)) > 0 \} \big] \geq c \, \E \big[ \# \mathcal{L}_r \big] ~.
\]
\end{proposition}

Proposition \ref{Propo:DensitySigma>0} contrasts with several Euclidean bi-dimensional trees \cite{ahlberghansonhoffman,coupier_sublin}, including the Euclidean RST \cite{baccellicoupiertran}, where the proportion of edges at level $r$ belonging to infinite paths is negligible. 

Proposition \ref{Propo:DensitySigma>0} is an immediate consequence of some intermediate results leading to Theorem \ref{Thm:positivevolume}, and its proof is done in Section \ref{Section:Densite>0}. It could be improved in several directions (an almost sure result rather than in expectation, a positive limit of the ratio rather than a positive $\liminf$ etc.) but this is not the goal of the current work.

\section{Proof of Theorem \ref{Thm:positivevolume}}
\label{Section:PositiveVolume}

\subsection{Global strategy}

Let $r_0\in (0,1)$, $u_0 \in \mathbb{S}^d$ and select the closest Poisson point to $(r_0 ; u_0)$ (with polar coordinates) denoted by
\[
z_0 := \argmin_{z\in\mathcal{N}} d(z,(r_0 ; u_0)) ~.
\]
Consider the event $E(r_0 ; u_0)$ on which the set of points at infinity $\mathcal{D}^\infty(z_0)$ reached by descendants of $z_0$ is nonempty but with null volume:
\begin{equation}
E(r_0 ; u_0) := \big\{ \mathcal{D}^\infty(z_0) \not= \emptyset \, \mbox{ and } \, \sigma \big( \mathcal{D}^\infty(z_0) \big) = 0 \big\} ~.
\end{equation}Our purpose is to prove that $\P(E(r_0;u_0))=0$.\\

For $\delta>0$, let us consider the set $\mathcal{G}(\delta)$ of descendants $z$ of $z_0$ whose ancestor $A(z)$ is not too close to $z$:
\[
\mathcal{G}(\delta) := \big\{ z \in \mathcal{D}(z_0) : \, d(z,A(z)) \geq \delta \big\} ~.
\]
Let $\mathcal{H}(\delta)$ be the set of corresponding radii:
\[
\mathcal{H}(\delta) := \big\{ r>0 : \, \exists u\in \mathbb{S}^d , (r;u) \in \mathcal{G}(\delta) \big\} ~.
\]
The next result states that, for $\delta$ small enough, if $z_0$ admits infinitely many descendants (or in an equivalent way $\mathcal{D}^\infty(z_0)\not=\emptyset$) then infinitely many of them are in $\mathcal{G}(\delta)$:

\begin{lemma}
\label{Lem:H-deltaUnbounded}
For any $\delta>0$ small enough, $\mathcal{D}^\infty(z_0)\not=\emptyset$ a.s. implies that $\mathcal{H}(\delta)$ is unbounded.
\end{lemma}

For any radius $r>0$, let $\mathcal{N}_{B(r)} := \mathcal{N}\cap B(r)$ be the PPP $\NN$ restricted to the ball $B(r)$. The following Lemma \ref{Lem:PositiveVolume} says that a.s. on $\mathcal{D}^\infty(z_0)\not=\emptyset$ the random variable $\mathbb{P}(\sigma(\mathcal{D}^\infty(z_0)) > 0 \,|\, \mathcal{N}_{B(r)})$ is  bounded away from $0$ for all radii $r\in\mathcal{H}(\delta)$ and uniformly on those radii.

\begin{lemma}
\label{Lem:PositiveVolume}
For any $\delta>0$ small enough, there exists $\eps>0$ such that, a.s. on $\mathcal{D}^\infty(z_0)\not=\emptyset$,
\begin{equation}
\forall r \in \mathcal{H}(\delta) , \, \mathbb{P} \big( \sigma(\mathcal{D}^\infty(z_0)) > 0 \,|\, \mathcal{N}_{B(r)} \big) \geq \eps ~.
\end{equation}
\end{lemma}

As shown below, Lemmas \ref{Lem:H-deltaUnbounded} and \ref{Lem:PositiveVolume} together imply $\P(E(r_0 ; u_0))=0$ from which Theorem \ref{Thm:positivevolume} immediatly follows. The proofs of Lemmas \ref{Lem:H-deltaUnbounded} and \ref{Lem:PositiveVolume} are respectively postponed to Sections \ref{Sect:H-deltaUnbounded} and \ref{Sect:PositiveVolume}.

\begin{proof}[Proof of Theorem \ref{Thm:positivevolume}]
Choose parameters $\delta$ and $\eps$ such that both Lemmas \ref{Lem:H-deltaUnbounded} and \ref{Lem:PositiveVolume} hold. The second lemma says that for this choice of $\delta $ and $\eps$, we have a.s. on $\{\mathcal{D}^\infty(z_0)\not= \emptyset \}$ and for any $r \in \mathcal{H}(\delta)$,
\[
\mathbb{P} \big( E(r_0 ; u_0) \,|\, \mathcal{N}_{B(r)} \big) \leq \mathbb{P} \big( \sigma(\mathcal{D}^\infty(z_0)) = 0 \,|\, \mathcal{N}_{B(r)} \big) \leq 1 - \eps ~.
\]
Since $\eps$ is uniform on $r \in \mathcal{H}(\delta)$ and since we work on $\{\mathcal{D}^\infty(z_0)\not= \emptyset\}$, the set $\mathcal{H}(\delta)$ is a.s. unbounded. We can take the $\liminf$: a.s. on $\{\mathcal{D}^\infty(z_0)\not= \emptyset\}$,
\begin{equation}
\label{LiminfE(r0)}
\liminf_{r\to\infty} \mathbb{P} \big( E(r_0 ; u_0) \,|\, \mathcal{N}_{B(r)} \big) \leq 1 - \eps ~.
\end{equation}
But the martingale convergence theorem asserts that:
\begin{equation}
\label{MartingaleCVTH}
\lim_{r \to \infty} \mathbb{P} \big( E(r_0 ; u_0) \,|\, \mathcal{N}_{B(r)} \big) = \ind_{E(r_0 ; u_0)}.
\end{equation}
So, on $E(r_0 ; u_0)\subset \{\mathcal{D}^\infty(z_0)\not=\emptyset\}$, the limit in \eqref{MartingaleCVTH} equals 1. Thus, both statements (\ref{LiminfE(r0)}) and (\ref{MartingaleCVTH}) are compatible only if $\mathbb{P}(E(r_0 ; u_0))=0$.

Hence, the union of the events $E(r_0 ; u_0)$ with rational radius $r_0$ and rational direction $u_0$ has null probability too. Since any Poisson point is the closest one to some rational element $(r_0 ; u_0)$ of $\mathbb{H}^{d+1}$, we can conclude that with probability $1$, any vertex $z \in \NN$ satisfies either $\mathcal{D}^\infty(z) = \emptyset$ (or $\mathcal{D}(z)$ is finite in an equivalent way) or $\sigma(\mathcal{D}^\infty(z))$ is (strictly) positive.
\end{proof}

\subsection{Proof of Lemma \ref{Lem:H-deltaUnbounded}}
\label{Sect:H-deltaUnbounded}

The proof of Lemma \ref{Lem:H-deltaUnbounded} is based on a percolation argument. To set it up, we first need a covering of the hyperbolic space $\mathbb{H}^{d+1}$ (except in the vicinity of the origin) with a uniform control of overlappings. This will be done using the Covering Lemma (below). This classical result will be used several times in this paper and is stated in \cite[Lemma 4.2]{flammant-RST}.

\begin{lemma}[Covering Lemma]
\label{Lem:covering}
There exists a covering constant $K=K(d)>0$ such that, for any {$r >0$}, there exists a collection of $N(r)$ points $z_1,...,z_{N(r)} \in S(r)$ such that:
\begin{itemize}
\item[(a)] $\bigcup_{1 \le i \le N(r)} B_{S(r)}(z_i , e^{-r}) = S(r)$,
\item[(b)] $\forall z \in S(r)$, $\# \{1 \le i \le N(r) : z \in B_{S(r)}(z_i,e^{-r})\} \leq K$.
\end{itemize}
Moreover, there exists $C=C(K,d)>0$ such that, for any {$r >0$}, $z \in S(r)$ and $A \ge 1$, the number of caps overlapping $B_{S(r)}(z , Ae^{-r})$ is bounded by $C A^d$:
\[
\# \Big\{ 1 \le i \le N(r) : \, B_{S(r)}(z_i,e^{-r}) \cap B_{S(r)}(z,Ae^{-r}) \not= \emptyset \Big\} \leq C A^d ~.
\]
In particular (taking $A := \pi e^r$ in the previous inequality), $N(r)\leq C e^{dr}$.
\end{lemma}

For any integer radius {$r >0$} and any index $1\leq m\leq N(r)$, let us set
\[
B_{r,m} := C(r,r+1) \cap \Cone(z_{r,m},e^{-r})
\]
where $\{z_{r,1},\ldots,z_{r,N(r)}\}$ is the collection of points of the sphere $S(r)$ given by Lemma \ref{Lem:covering} (in this proof we stress the dependence on $r$ of the $z_m$'s by adding the index $r$ in $z_{r,m}$). Each block $B_{r,m}$ is based on the spherical cap $B_{S(r)}(z_{r,m},e^{-r}) = S(r)\cap\Cone(z_{r,m},e^{-r})$ and has thickness $1$. Two blocks $B_{r,m}$ and $B_{r',m'}$ are said to be \textit{adjacent}, which is denoted by $B_{r,m} \sim B_{r',m'}$, if they are at distance  less than $\delta$ from each other. Lemma \ref{Lem:Block} asserts that the volumes of $B_{r,m}$'s are uniformly bounded and the degrees in the graph generated by the adjacency relation are bounded.

\begin{lemma}
\label{Lem:Block}
There exist two positive constants $C_1$ and $C_2$ only depending on $d$ such that for all integer radius {$r>0$} and $m\in \{1,\ldots,N(r)\}$, the following holds :
\begin{equation}
\label{eq:Brm}
\Vol(B_{r,m}) \leq C_1 , \quad \mbox{ and } \quad \# \big\{ (r',m') : \, B_{r',m'} \sim B_{r,m} \big\} \leq C_2 ~.
\end{equation}
\end{lemma}

\begin{proof}
First, using \eqref{area_cap},
\begin{eqnarray*}
\Vol(B_{r,m}) & = & \int_r^{r+1} \sinh^d(\rho)  \times \sigma \big(\Cone(z_{\rho,m},e^{-\rho}) \cap \mathbb{S}^d \big) \, d\rho \\
& \leq & {C}\int_r^{r+1} e^{d\rho} \, d\rho \times e^{-dr} = \frac{{C}}{d} (e^d - 1) =: C_1 ~.
\end{eqnarray*}
For the second inequality of \eqref{eq:Brm}, two blocks $B_{r',m'}$ and $B_{r,m}$ are adjacent whenever $B_{r',m'}$ overlaps $\{z \in \H^{d+1} : d(z,B_{r,m}) \leq \delta\}$. Since $\delta<1$ this forces $r'$ to be in $\{r-1,r,r+1\}$ and $B_{r',m'}$ to overlap $\Cone(z_{r,m},(1+\delta)e^{-r})$. 
Now, Lemma \ref{Lem:covering} asserts that for each layer $r' \in \{r-1,r,r+1\}$, there are at most $C A^d= C (1+\delta)^d e^{d(r'-r)} \leq C 2^d e^d$ blocks $B_{r',m'}$ overlapping $\Cone(z_{r,m},(1+\delta)e^{-r})$. Consequently, $B_{r,m}$ has at most $C_2 := 3 C 2^d e^d$ adjacent blocks.
\end{proof}

The block $B_{r,m}$ is said \textit{$\delta$-bad} if it contains a Poisson point $z$ such that $d(z,A(z))<\delta$. Let us first prove that it is possible to choose $\delta>0$ small enough so that the set of $\delta$-bad blocks denoted as $\Psi_\delta$ is a.s. subcritical w.r.t. the adjacency relation, i.e. $\Psi_\delta$ only admits finite connected components. To do it, we first use the Mecke's formula \cite[Prop. 13.1.IV]{daleyverejones} to bound the probability for a block to be $\delta$-bad: for any $(r,m)$,
\begin{eqnarray}
\label{p(delta)}
\P \big(B_{r,m} \mbox{ is $\delta$-bad}\big) & \leq & \E \left[ \# \{z \in B_{r,m} \cap \mathcal{N} : \, d(z,A(z)) < \delta\} \right] \nonumber \\
& = & \E \left[ \# \{ z \in B_{r,m} \cap \mathcal{N} : \, B^+(z,\delta) \cap \mathcal{N} \not= \emptyset \right] \nonumber \\
& = & \lambda \int_{B_{r,m}} \P \big( B^+(z,\delta) \cap \mathcal{N} \not= \emptyset \big) \, dz \nonumber \\
& \leq & \lambda \Vol(B_{r,m}) \big( 1-e^{-\lambda \Vol(B(\delta))} \big) \nonumber \\
& \leq & \lambda C_1 \big( 1-e^{-\lambda \Vol(B(\delta))} \big) =: p(\delta)
\end{eqnarray}
thanks to Lemma \ref{Lem:Block}. Hence, $\P(B_{r,m}\mbox{ is $\delta$-bad})$ tends to $0$ as $\delta\to 0$ uniformly on the couple $(r,m)$.

In a second step, we adapt the Peierls argument to our context to establish the subcriticality of $\Psi_\delta$ for $\delta$ small enough. Given $(r,m)$, let $\mathcal{P}_{r,m}(k)$ be the set of paths with length $k$ of adjacent blocks starting at $B_{r,m}$. By Lemma \ref{Lem:Block}, $\# \mathcal{P}_{r,m}(k) \leq (C_2)^{k}$. Such a path $\pi=(B_1,\ldots,B_k)$ is said $\delta$-\textit{bad} if all the blocks $B_i$ it contains are $\delta$-bad. Henceforth the probability for $B_{r,m}$ to belong to an infinite connected component of $\delta$-bad blocks is upperbounded by
\[
\limsup_{k\to\infty} \sum_{\pi \in \mathcal{P}_{r,m}(k)} \P \big( \pi \mbox{ is $\delta$-bad} \big) ~.
\]
For any path $\pi=(B_1,\ldots,B_k)$ in $\mathcal{P}_{r,m}(k)$, we can choose a subset of blocks $\{B_{i_1},\ldots,B_{i_\ell}\}$ included in $\{B_1,\ldots,B_k\}$ such that the $B_{i_j}$'s are two by two non adjacent and $\ell\geq k/C_2$ (here we use that each block has at most $C_2$ adjacent blocks). Besides, the adjacency relation has been defined so that the events $\{B_{i_j} \mbox{ is $\delta$-bad}\}$'s are mutually independent. So
\[
\P \big( \pi \mbox{ is $\delta$-bad} \big) \leq \P \big( B_{i_1},\ldots,B_{i_\ell} \mbox{ are $\delta$-bad} \big) \leq p(\delta)^{k/C_2} ~,
\]
where $p(\delta)$ is defined in (\ref{p(delta)}). It follows
\[
\sum_{\pi \in \mathcal{P}_{r,m}(k)} \P \big( \pi \mbox{ is $\delta$-bad} \big) \leq \big( C_2 p(\delta)^{1/C_2} \big)^k ~. 
\]
Choosing $\delta$ small enough so that $C_2 p(\delta)^{1/C_2}<1$, we then obtain that a.s. the block $B_{r,m}$ cannot belong to an unbounded connected component of $\delta$-bad blocks. Consequently, for such parameter $\delta$, the set $\Psi_\delta$ is subcritical with probability $1$.\\

To conclude the proof of Lemma \ref{Lem:H-deltaUnbounded}, let us pick $\delta$ small enough such that the set of $\delta$-bad blocks does not percolate. Assume also that $\mathcal{D}^\infty(z_0)$ is non empty, i.e. the subtree of the RST rooted at $z_0$ admits (at least) one infinite path $(z_n)_{n\geq 0}$ of Poisson points. The PPP $\mathcal{N}$ being locally finite, the path $(z_n)_{n \geq 0}$ cannot be stuck, from some index, inside a $\delta$-bad connected component of blocks which is bounded by choice of $\delta$. It eventually comes out of each $\delta$-bad connected component. Two cases must be distinguished.\\
Either $(z_n)_{n\geq 0}$ visits infinitely many $\delta$-good blocks where of course a block is said \textit{$\delta$-good} if it is not $\delta$-bad. Hence, infinitely many of the $z_n$'s satisfy $d(z_n,A(z_n))\geq \delta$. These vertices are in $\mathcal{G}(\delta)$ which means that $\mathcal{H}(\delta)$ is unbounded.\\
Or, the infinite path $(z_n)_{n\geq 0}$ jumps infinitely many times from a $\delta$-bad connected component to another one. But, by construction, two different bad connected components are at distance at least $\delta$ from each other. So these jumps provide as many $z_n$'s in $\mathcal{G}(\delta)$: $\mathcal{H}(\delta)$ is unbounded in this case too.

\subsection{Proof of Lemma \ref{Lem:PositiveVolume}}
\label{Sect:PositiveVolume}

In what follows, we will modify configurations locally. This is why we emphasize the dependence of any random set $\mathcal{A}$ (think about $\mathcal{D}(z)$, $\mathcal{D}^\infty(z)$ or $\mathcal{H}(\delta)$) on the current configuration $\eta$, a realization of the PPP $\NN$, by writting $\mathcal{A}[\eta]$.

Given $r_0>0$ and $u_0 \in \mathbb{S}^d$, recall that $z_0$ is the closest Poisson point to $(r_0 ; u_0)$. For this section, let us consider a configuration $\eta$ satisfying $\mathcal{D}^\infty(z_0)\not= \emptyset$. Let $\delta>0$ and $r>r_0$ be an element of the set $\mathcal{H}(\delta)[\eta]$-- by Lemma \ref{Lem:H-deltaUnbounded}, this latter set is unbounded provided $\delta$ is small enough. Let us set $z_1 := (r;u) \in \mathcal{G}(\delta)$ for some $u\in \mathbb{S}^d$ ($z_1$ is in $\NN$). In the sequel, we will work conditionally on $\NN_{B(r)} = \eta_{B(r)}$, the configuration of the PPP $\NN$ restricted to the ball $B(r)$.

Let $0<\delta'<\delta$ and $h>0$ (thought as large). Let also $z_2 := (r+h;u)$ (with the same direction as $z_1$). The event $\Stab(r,h,\delta')$ encodes the fact that the set of descendants of any Poisson point in $B(z_2,\delta')$ is not sensitive to what happens inside $B(r)$:
\begin{equation}\label{def:Stab}
\Stab(r,h,\delta') := \big\{ \eta' : \, \forall z \in \NN \cap B(z_2,\delta') , \, \forall \eta'' , \, \mathcal{D}(z)[\eta''_{B(r)} \cup \eta'_{B(r)^\complement}] = \mathcal{D}(z)[\eta'] \big\} ~.
\end{equation}
It will be proved in Lemma \ref{Lem:LimitDeter} that $\Stab(r,h,\delta')$ has a probability tending to $1$ as $h\to\infty$ uniformly on $r$.

Let us now introduce the event $F(r,h,\delta')$ on which there exists $z \in \NN \cap B(z_2,\delta')$ such that the four following items hold:
\begin{itemize}
\item[$(i)$] $\NN \cap B(z_2,\delta') = \{z\}$,
\item[$(ii)$] $\sigma(\mathcal{D}^\infty(z)) > 0$,
\item[$(iii)$] $\Stab(r,h,\delta')$ occurs
\item[$(iv)$] and $\mathcal{D}(z) \subset B(r+h+\delta')^\complement$.
\end{itemize}
On the event $F(r,h,\delta')$, there is a unique point of the PPP $\NN$ in the ball $B(z_2,\delta')$ that has a thick trace at infinity. This subtree is outside the ball $B(r+h+\delta')$ and does not depend on the points of $\NN$ in the ball $B(r)$. The distance $h$ is the separation gap ensuring that the subtree rooted at $z_2=(r+h;u)$ remains independent from the Poisson points inside the ball $B(r)$.

Lemma \ref{Lem:unifF0} states that the event $F(r,h,\delta')$ has a probability larger than some positive $\eps$ which is uniform on $r$. Getting such uniformity of $\eps$ (or $A$) on $r$ will significantly complicate the proof of Lemma \ref{Lem:unifF0}, postponed to Section \ref{Section:UnifF0}.

\begin{lemma}
\label{Lem:unifF0}
There exists $A>0$ large enough such that, for any $\delta'>0$ small enough, there exists $\eps=\eps(A,d,\lambda,\delta')>0$ such that for any $h_0>0$ large enough and $r>0$, there exists $h \in [h_0,h_0+A]$ such that $\P( F(r,h,\delta') ) \geq \eps$. Note also that among the previous parameters, only $h$ may depend on $r$.
\end{lemma}

For $\theta>0$, let us consider the subset $U = U(r,h,\delta',\theta)$ of $\mathbb{H}^{d+1}$ defined as
\begin{equation}
\label{def:U}
U := \Big( B(r+h+\delta') \setminus (B(r) \cup B(z_2,\delta')) \Big) \cap \Cone(0,\theta) ~.
\end{equation}
See Figure \ref{fig:ConstructionU}. Let $\eta' \in F(r,h,\delta')$ such that $\eta'_{B(r)}=\eta_{B(r)}$. When all the Poisson points of the configuration $\eta'$ inside the set $U$ are removed, the ancestor of $z$-- the only Poisson point in $B(z_2,\delta')$ according to the event $F(r,h,\delta')$ --becomes $z_1$ which itself is a descendant of $z_0$. This leads to $\sigma(\mathcal{D}^{\infty}(z_0))[\eta'_{U^\complement}] > 0$. This construction requires the hypothesis $r\in\mathcal{H}(\delta)[\eta]$, i.e. there is no Poisson points in $B^+(z_1,\delta)$, and this is the only place in the proof of Lemma \ref{Lem:PositiveVolume} where it is needed. 

\begin{lemma}
\label{Lem:emptying}
For any $r\in\mathcal{H}(\delta)$ with $r\geq r_0$, for any $h$ and $\delta'<\delta/3$, there exists $\theta$ large enough such that the following statement holds. Almost every configuration $\eta'$ with $\eta'_{B(r)}=\eta_{B(r)}$ and $\eta'\in F(r,h,\delta')$ satisfies $\sigma(\mathcal{D}^{\infty}(z_0))[\eta'_{U^\complement}] > 0$.
\end{lemma}

\begin{figure}[!ht]
\begin{center}
\psfrag{0}{\small{$0$}}
\psfrag{a}{\small{$z_0$}}
\psfrag{b}{\small{$z_1$}}
\psfrag{c}{\small{$z$}}
\psfrag{d}{\small{$B(z_2,\delta')$}}
\psfrag{e}{\small{$r$}}
\psfrag{f}{\small{$r+h+\delta'$}}
\psfrag{g}{\small{$\infty$}}
\psfrag{h}{\small{$r_0$}}
\includegraphics[width=13cm,height=6.5cm]{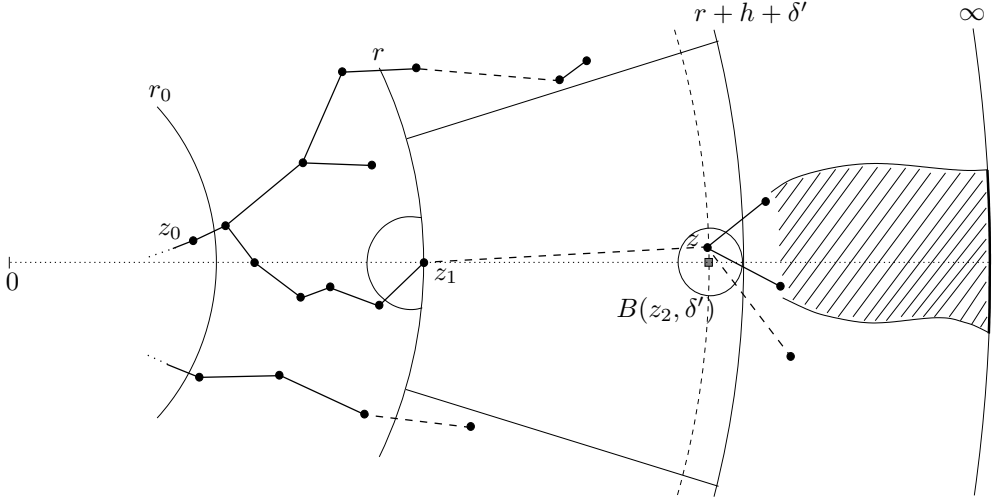}
\caption{\label{fig:ConstructionU}Illustration of Lemma \ref{Lem:emptying}. Black dots are Poisson points. This picture represents a configuration $\eta'_{U^\complement}$ (i.e. Poisson points of $\eta'_U$ have been removed) where $\eta'_{B(r)}=\eta_{B(r)}$ and $\eta'\in F(r,h,\delta')$. The ball $B(z_2,\delta')$ whose (deterministic) center $z_2$ is marked by a grey square, contains only one Poisson point, namely $z$. Its set of descendants $\mathcal{D}(z)$ is represented by the hatched region and satisfies $\sigma(\mathcal{D}^{\infty}(z))[\eta'] > 0$ (the bold black curve on the unit sphere $\mathbb{S}^d$). When the set $U$ is emptying of Poisson points then $A(z)[\eta'_{U^\complement}]=z_1$ which means that $\sigma(\mathcal{D}^{\infty}(z_0))[\eta'_{U^\complement}] > 0$. Edges in dotted lines have been modified after deleting $\eta'_{U}$.}
\end{center} 
\end{figure}

We are now ready to prove Lemma \ref{Lem:PositiveVolume}.

\begin{proof}[Proof of Lemma \ref{Lem:PositiveVolume}]
Parameters $A$, $\delta'$, $\eps$ and $h_0$ are chosen according to Lemma \ref{Lem:unifF0}. Then for any $r>0$ there exists $h\in [h_0,h_0+A]$ such that $\P(F(r,h,\delta')) \geq \eps$. Because it requires $\Stab(r,h,\delta')$, the event $F(r,h,\delta')$ does not depend on the configuration inside the ball $B(r)$. So,
\[
\P \big( F(r,h,\delta') \,|\, \NN_{B(r)}=\eta_{B(r)} \big) = \P \big( F(r,h,\delta') \big) \geq \eps ~.
\]
Now $r$ can be chosen in $\mathcal{H}(\delta)$ with $r>r_0$ and $0<\delta'<\delta/3$ so that Lemma \ref{Lem:emptying} applies:
\begin{eqnarray*}
\eps & \leq & \P \big( \eta' \in F(r,h,\delta') \,|\, \NN_{B(r)}=\eta_{B(r)} \big) \\
& \leq & \P \big( \sigma(\mathcal{D}^{\infty}(z_0))[\eta'_{U^\complement}] > 0 \,|\, \NN_{B(r)}=\eta_{B(r)} \big) \\
& = & \P \big( \sigma(\mathcal{D}^{\infty}(z_0))[\eta'] > 0 , \, \eta'_U = \emptyset \,|\, \NN_{B(r)}=\eta_{B(r)} \big) \\
& \leq & \P \big( \sigma(\mathcal{D}^{\infty}(z_0)) > 0 \,|\, \NN_{B(r)}=\eta_{B(r)} \big) ~.
\end{eqnarray*}
Note that the lower bound $\eps=\eps(A,d,\lambda,\delta')$ does not depend on the parameter $r$.
\end{proof}

This section ends with the proof of Lemma \ref{Lem:emptying}.

\begin{proof}[Proof of Lemma \ref{Lem:emptying}]
Consider a configuration $\eta'$ equal to $\eta$ inside the ball $B(r)$ and belonging to the event $F(r,h,\delta')$. Let us first assume that, for the configuration $\eta'_{U^\complement}$, the ancestor of the Poisson point $z$ (whose existence is given by $F(r,h,\delta')$) is $z_1$, i.e.
\begin{equation}
\label{AncestorZ}
A(z)[\eta'_{U^\complement}] = z_1 ~.
\end{equation}
Let us prove $\Vol(\mathcal{D}^{\infty}(z_0))[\eta'_{U^\complement}] > 0$ from (\ref{AncestorZ}). Since $\eta'\in F(r,h,\delta')$, the set of descendants $\mathcal{D}(z)[\eta']$ is included in $B(r+h+\delta')^\complement$. So, removing Poisson points of $\eta'_U$ modifies no edges $(z',A(z'))$ of the RST as long as $A(z')\in\mathcal{D}(z)$. In fact, removing $\eta'_U$ may only add new descendants to the vertex $z$. Hence, $\mathcal{D}(z)[\eta']$ is included in $\mathcal{D}(z)[\eta'_{U^\complement}]$ which leads to $\sigma(\mathcal{D}^{\infty}(z))[\eta'_{U^\complement}] \geq \sigma(\mathcal{D}^{\infty}(z))[\eta'] >0$. Finally, using $A(z)[\eta'_{U^\complement}]=z_1$ which is itself a descendant of $z_0$, we get $\sigma(\mathcal{D}^{\infty}(z_0))[\eta'_{U^\complement}] > 0$.

It then remains to show (\ref{AncestorZ}). Let $R$ be the hyperbolic distance between $z$ and $z_1$. It is sufficient to prove that $B^+(z,R)$ is included in $U\cup B^+(z_1,\delta)$ since, after removing Poisson points in $U$, $z_1$ would become the closest Poisson point to $z$. Here we use that $r\in\mathcal{H}(\delta)$, i.e. the ball $B^+(z_1,\delta)$ contains no other Poisson points except $z_1$. Let us first remark that
\[
R = d(z,z_1) \leq d(z,z_2) + d(z_2,z_1) \leq \delta' + h \leq h + 1
\]
with $\delta'<1$. So, taking $\theta$ large enough such that $B^+(z,R)\subset \Cone(0,\theta)$, it then remains to prove that 
\begin{equation}\label{etape1}
V:=B^+(z,R)\cap B(r) \quad \subset \quad B^+(z_1,\delta). 
\end{equation}
To do so, let us pick $v_1, v_2, v_3$ on the geodesic $\gamma$ between $0$ and $z$ as follows: $v_1$, $v_2$ and $v_3$ are the intersection points between $\gamma$ and resp. the sphere $S(0,r+h)$, $S(0,r)$ and $S(z,R)$. Let us denote by $w$ the symmetric of $z_1$ w.r.t. the line $(0z)$. Henceforth, it is sufficient to show that $w$, $z_1$ and $v_3$ belong to $B^+(z_1,\delta)$. First, we have $d(v_2,z_1)\leq d(v_1,z_2)\leq \delta'$. Thus, by symmetry, $d(w,z_1)\leq d(w,v_2)+d(v_2,z_1)=2d(v_2,z_1)\leq 2\delta'$. Since $v_3$, $v_2$ and $z$ are on the same geodesic, we can write:
\[
d(v_3,v_2) = d(v_3,z) - d(z,v_2) \leq R - h + \delta' ~.
\]
Since $R\leq h+\delta'$, we get $d(v_3,v_2)\leq 2\delta'$ and then $d(v_3,z_1)\leq d(v_3,v_2)+d(v_2,z_1)\leq 3\delta'$. Whenever $\delta' <\delta/3$, the three points $w$, $z_1$ and $v_3$ are inside $B^+(z_1,\delta)$. This proves \eqref{etape1} and concludes the proof.
\end{proof}

\section{Proof of Lemma \ref{Lem:unifF0}: uniformity in $r$}\label{Section:UnifF0}

This section is devoted to the proof of Lemma \ref{Lem:unifF0}, i.e. $\P( F(r,h,\delta') ) \geq \eps$ where the lower bound $\eps$ does not depend on $r,h$. Recall the notation of Section \ref{Sect:PositiveVolume} and of Figure \ref{fig:ConstructionU}. We first prove in Section \ref{sect:PositiveDensity>0} that a positive proportion of Poisson points of $B(z_2,\delta')$ satisfies $\sigma(\mathcal{D}^\infty(\cdot)) > 0$-- Item $(ii)$ in the definition of $F(r,h,\delta')$ --where $z_2=(r+h;u)$. Then we prove that the properties described by the other three items occur with high probability. In particular we state in Section \ref{sect:Deterto1} that $\Stab(r,h,\delta')$ has a probability tending to $1$ as $h\to\infty$ uniformly on $r$. But first, in Section \ref{Appendix:MBD}, we recall some properties of the Maximal Backward angular Deviations (MBD) which is the key tool here to control the path fluctuations in the hyperbolic RST.

\subsection{Maximal backward angular deviations}
\label{Appendix:MBD}

A crucial ingredient in this work is the control of \textit{Maximal Backward angular Deviations} (MBD) which has been done in \cite{flammant-RST} (see also \cite{HyperbolicDSF} where these notions are introduced). Let us recall here the main definitions and results.

Given $r>0$ and $z\in \mathcal{L}_r$, recall that $z_{\downarrow}$ denotes the Poisson point whose arc $|\![z_{\downarrow}, A(z_{\downarrow})]\!|$ crosses $S(r)$ at $z$ and set $z_{\uparrow}:=A(z_{\downarrow})$. Let also $A^{(k)}(z) := A\circ\ldots\circ A(z)$ ($k$ times) be the $k$-th ancestor of $z$ for any $k\in\mathbb{N}$ (by convention, we set $A(0):=0$).

Consider $0<r \leq r'$ and $z'\in\mathcal{L}_{r'}$. Let us denote by $z\in\mathcal{L}_{r}$ the intersection point between the path of $\RST$ joining $z'$ to the origin and the sphere $S(r)$. Let us define $\CFD_r^{r'}(z')$ as the \textit{Cumulative Forward angular Deviations between levels $r$ and $r'$} as
\begin{eqnarray*}
\CFD_r^{r'}(z'):=
\left\{
\begin{aligned}[l]
&\widehat{z' 0 z} \text{ if } z_\downarrow=z'_\downarrow,\\
&\widehat{z' 0 z'_\uparrow}+\sum_{k=0}^{n-1}\widehat{A^{(k)}(z'_\uparrow) 0 A^{(k+1)}(z'_\uparrow)}+\widehat{z_\downarrow 0 z} \text{ else,}
\end{aligned}
\right.
\end{eqnarray*}
where $n$ is the unique non negative integer such that $A^{(n)}(z'_\uparrow)=z_\downarrow$. We also set $\CFD_r^{r'}(z'):=0$ when $z' \notin \mathcal{L}_{r'}$.

\begin{definition}[Maximal Backward angular Deviations]
\label{def:MBD}
For $0<r \leq r'$, let us define the \textit{Maximal Backward angular Deviations between levels $r$ and $r'$} as
\begin{equation}
\label{MBD}
\MBD_r^{r'}(z) := \sup_{\rho \in [r,r']} \max_{y \in \mathcal{D}_{r}^{\rho}(z)} \CFD_{r}^{\rho}(y)
\end{equation}
if $z \in \mathcal{L}_{r}$ and $\MBD_r^{r'}(z):=0$ if $z \notin \mathcal{L}_{r}$, where $\mathcal{D}_{r}^{\rho}(z)$ is the set of points $y\in\mathcal{L}_\rho$ whose path of $\RST$ from $y$ to $0$ cuts $S(r)$ at $z$ (or, roughly speaking, the set of descendants of $z$ at level $\rho$).
\end{definition}

Since $r' \mapsto \MBD_r^{r'}(z)$ is non-decreasing, the MBD can be naturally extend to $r'=\infty$ by setting:
\[
\MBD_r^\infty(z) := \lim_{r' \to \infty} \MBD_r^{r'}(z) ~.
\]
The next lemma provides a control of the moments of $\MBD_r^\infty(z)$, which is crucial for proving the positive density. The idea is that because the paths do not fluctuate too much, there is room for a positive number of points at a given level $r$ to have a thick trace at infinity.

\begin{lemma}[Proposition 2.6 of \cite{flammant-RST}]
\label{lem:MBD}
For any $p\geq 3d/2$, there exists a constant $C=C(d)>0$ such that, for any $r>2$, $A>0$ and any direction $u \in \mathbb{S}^d$,
\begin{equation}
\mathbb{E} \Big[ \sum_{z \in B_{S(r)}(u , A e^{-r}) \cap \RST} \big( \MBD_{r}^\infty(z) \big)^p \Big] \leq  C A^d e^{-rp} ~.
\end{equation}
\end{lemma}

\subsection{Positive density of $\sigma(\mathcal{D}^\infty(\cdot))>0$}
\label{sect:PositiveDensity>0}

\begin{lemma}
\label{Lem:PositiveDensity}
There exists $A>0$ large and $c_0=c_0(A,d,\lambda)>0$ such that for any $0<\delta'<1$, $h_0>0$ and $r>0$, there exists $h \in [h_0,h_0+A]$ such that for any $z_2=(r+h;u)$ ($u\in \mathbb{S}^d$)
\begin{equation}
\label{PositiveDensity}
\E \Big[ \# \big\{ z \in \NN \cap B(z_2,\delta') : \, \sigma(\mathcal{D}^\infty(z)) > 0 \big\} \Big] \geq c_0 \Vol \big( B(\delta') \big) ~.
\end{equation}
Note also that among the previous parameters, only $h=h(r)$ may depend on $r$.
\end{lemma}

\begin{proof}[Proof of Lemma \ref{Lem:PositiveDensity}] The proof is splitted into three steps. We first start with proving an estimate close to \eqref{PositiveDensity}, but for $z\in \mathcal{L}_r$ (Step 1). Then, we extend the result to portion of annuli and to balls.\\

\noindent
\textbf{Step 1.} Let us first prove that there exists $c=c(d)>0$ such that for any {$r> 0$},
\begin{equation}
\label{densite>0-Lr}
\E \Big[ \# \big\{ z \in \mathcal{L}_{r} : \, \sigma(\mathcal{D}^{\infty}(z)) > 0 \big\} \Big] \geq c e^{dr} ~.
\end{equation}

Let us first use the Cauchy-Schwarz inequality with the inner product $\langle X,Y \rangle := \E[\sum_i X_i Y_i]$:
\begin{eqnarray}
\label{Density>0-CS}
\E \Big[ \sum_{z \in \mathcal{L}_{r}} \sigma(\mathcal{D}^{\infty}(z)) \Big]^2 & \leq & \E \Big[ \sum_{z \in \mathcal{L}_{r}} \ind_{\sigma(\mathcal{D}^{\infty}(z))>0} \Big] \E \Big[ \sum_{z \in \mathcal{L}_{r}} \sigma(\mathcal{D}^{\infty}(z))^2 \Big] \nonumber \\
& = & \E \Big[ \# \big\{ z \in \mathcal{L}_{r} : \, \sigma(\mathcal{D}^{\infty}(z)) > 0 \big\} \Big] \E \Big[ \sum_{z \in \mathcal{L}_{r}} \sigma(\mathcal{D}^{\infty}(z))^2 \Big] ~.
\end{eqnarray}
Thus, the left hand side of \eqref{densite>0-Lr} is lower bounded by
\[\E \Big[ \# \big\{ z \in \mathcal{L}_{r} : \, \sigma(\mathcal{D}^{\infty}(z)) > 0 \big\} \Big]\geq \frac{\E \Big[ \sum_{z \in \mathcal{L}_{r}} \sigma(\mathcal{D}^{\infty}(z)) \Big]^2}{\E \Big[ \sum_{z \in \mathcal{L}_{r}} \sigma(\mathcal{D}^{\infty}(z))^2 \Big]}.\]

Recall that in the open-ball model, the set of points at infinity $\partial \H^{d+1}$ is identified with the unit $d$-dimensional sphere $\mathbb{S}^d$ whose spherical probability measure $\sigma(\mathbb{S}^d)$ is (normalized to) $1$. By \cite[Theorem 1.1 (i)]{flammant-RST}, with probability $1$, any point at infinity is the asymptotic direction of (at least) one infinite path of the hyperbolic RST. So,
\begin{equation}
\label{CSQ-Th1.1(i)}
\mbox{a.s. } \, 1 = \sigma(\mathbb{S}^d) \leq \sum_{z \in \mathcal{L}_{r}} \sigma(\mathcal{D}^{\infty}(z)) ~.
\end{equation}
Given {$r>0$} and $z\in\mathcal{L}_r$, let us denote by $\textrm{Ang}(z):=\sup\{\widehat{z0I} : I \in \mathcal{D}^{\infty}(z)\}$ if $\mathcal{D}^{\infty}(z)\not=\emptyset$ and $\textrm{Ang}(z):=0$ otherwise. Then, the set $\mathcal{D}^{\infty}(z)$ is included in the spherical cap $\Cone(z,\textrm{Ang}(z)) \cap \partial \H^{d+1}$, which means $\sigma(\mathcal{D}^{\infty}(z)) \leq C \textrm{Ang}(z)^d$. Moreover, the quantity $\textrm{Ang}(z)$ is bounded by the Maximal Backward angular Deviation $\MBD_{r}^{\infty}(z)$: see Definition \ref{def:MBD}. We then get
\begin{equation}
\label{Ang-MBD}
\E \Big[ \sum_{z \in \mathcal{L}_{r}} \sigma(\mathcal{D}^{\infty}(z))^2 \Big] \leq \E \Big[ \sum_{z \in \mathcal{L}_{r}} \MBD_{r}^{\infty}(z)^{2d} \Big] \leq C e^{-dr}
\end{equation}
by Lemma \ref{lem:MBD} (applied with $A=\pi e^{r}$ and $p=2d$), where the constant $C=C(d)>0$ does not depend on $r>0$.

Finally, \eqref{densite>0-Lr} follows from combining \eqref{Density>0-CS}, \eqref{CSQ-Th1.1(i)} and \eqref{Ang-MBD}.\\

Inequality \eqref{densite>0-Lr} refers to the elements of $\mathcal{L}_r$ while \eqref{PositiveDensity} concerns Poisson points. Passing from ones to others while preserving the uniformity of $A$ in $r$ requires some technical considerations. We now extend \eqref{densite>0-Lr} to annuli and then to balls.\\ 

\noindent
\textbf{Step 2.} For any $A$ large enough there exists $c_0=c_0(d,A)>0$ such that for any $r>0$ and $0<\delta'<1$,
\begin{equation}
\label{Densite>0inC}
\E \Big[ \# \big\{ z \in \NN \cap C(r,r+A-2\delta') : \, \sigma(\mathcal{D}^{\infty}(z)) > 0 \big\} \Big] \geq c_0 \Vol\big( C(r,r+A-2\delta') \big) ~,
\end{equation}where we recall that $C(r,R)$ has been defined in the notations of Section \ref{Section:hyp_geom}.

For $z\in \mathcal{L}_r$, recall that $z_{\downarrow}$ denotes the Poisson point whose arc $|\![z_{\downarrow}, A(z_{\downarrow})]\!|$ crosses $S(r)$ in $z$. In particular, $\sigma(\mathcal{D}^{\infty}(z))$ and $\sigma(\mathcal{D}^{\infty}(z_{\downarrow}))$ are equal. Henceforth,
\begin{eqnarray}
\lefteqn{\E \Big[ \# \big\{ z \in \NN \cap C(r,r+A-2\delta') : \, \sigma(\mathcal{D}^{\infty}(z)) > 0 \big\} \Big]} \nonumber\\
& &\hspace*{1cm} \geq \E \Big[ \# \big\{ z \in \mathcal{L}_r : \, z_{\downarrow} \in B(r+A-2\delta') \, \mbox{ and } \, \sigma(\mathcal{D}^{\infty}(z)) > 0 \big\} \Big] \nonumber\\
& &\hspace*{1cm} \geq c_1 \E \Big[ \# \big\{ z \in \mathcal{L}_r : \, \sigma(\mathcal{D}^{\infty}(z)) > 0 \big\} \Big].\label{etape2}
\end{eqnarray}
The inequality \eqref{etape2} makes the object of Lemma \ref{Lem:Densite1/2}, and its proof is postponed to Section \ref{Sect:TechLemmas}. This choice is done as the proof uses arguments similar to the ones developed in the next Section in a more difficult case. The constants $c_1=c_1(d)>0$ and $A$ above are chosen large enough according to Lemma \ref{Lem:Densite1/2}. It then remains to use Step 1 and the inequalities
\[
\Vol \big( C(r,r+A-2\delta') \big) \leq \Vol \big( B(r+A) \big) \leq C e^{d(r+A)} ~,
\]
to finally get
\begin{multline*}
\E \Big[ \# \big\{ z \in \NN \cap C(r,r+A-2\delta') : \, \sigma(\mathcal{D}^{\infty}(z)) > 0 \big\} \Big]\\
\begin{aligned} \geq & c_1 \ c\ e^{dr}\\
\geq & c_1 \times c \times (C e^{dA})^{-1} \Vol \big( C(r,r+A-2\delta') \big) ~.
\end{aligned}
\end{multline*}

\noindent
\textbf{Step 3.} For short, let us set $f(z'):=\P(\sigma(\mathcal{D}^{\infty}(z'))>0 \,|\, z'\in\NN)$. The Mecke's formula and Fubini's theorem allow us to write:
\begin{eqnarray}
\label{Campbell*2&Fubini}
\lefteqn{\int_{C(r,r+A)} \E \Big[ \# \big\{ z' \in \NN \cap B(z,\delta') : \, \sigma(\mathcal{D}^{\infty}(z')) > 0 \big\} \Big] \, dz} \nonumber\\
& &\hspace*{2cm} = \lambda \int_{z\in C(r,r+A)} \int_{z'\in B(z,\delta')} f(z') \, dz' \, dz \nonumber\\
& &\hspace*{2cm} = \lambda \int_{z'\in C(r-\delta',r+A+\delta')} f(z') \Vol\big( B(z',\delta')\cap C(r,r+A) \big) \, dz' \nonumber\\
& &\hspace*{2cm} \geq \lambda \Vol\big( B(\delta') \big) \int_{z'\in C(r+\delta',r+A-\delta')} f(z') \, dz' \nonumber\\
& &\hspace*{2cm} \geq \lambda \Vol\big( B(\delta') \big) \E \Big[ \# \big\{ z' \in \NN \cap C(r+\delta',r+A-\delta') : \, \sigma(\mathcal{D}^{\infty}(z')) > 0 \big\} \Big] \nonumber\\
& &\hspace*{2cm} \geq \lambda \Vol\big( B(\delta') \big) c_0 \Vol\big( C(r+\delta',r+A-\delta') \big)
\end{eqnarray}
by Step 2.

Now, using Inequalities (\ref{InegVol}), it is not difficult to choose $\tilde{c}=\tilde{c}(d)>0$ (small) and $A=A(d)>0$ large enough and uniform on $r>0$ and $\delta'<1$ such that $\Vol(C(r+\delta',r+A-\delta'))$ is bigger than $\tilde{c} \Vol(C(r,r+A))$. Combining with \eqref{Campbell*2&Fubini}, we get
\[
\int_{C(r,r+A)} \E \Big[ \# \big\{ z' \in \NN \cap B(z,\delta') : \, \sigma(\mathcal{D}^{\infty}(z')) > 0 \big\} \Big] \, dz \geq c_2 \Vol\big( B(\delta') \big) \Vol\big( C(r,r+A) \big) ~.
\]
with $c_2 := \lambda c_0 \tilde{c}$. This forces the existence of some $(r+h ; u) \in C(r,r+A)$, with $h \in [0,A]$ and $u \in \mathbb{S}^d$, satisfying
\begin{equation}
\label{Fin-c2}
\E \Big[ \# \big\{ z' \in \NN \cap B( (r+h ; u) ,\delta') : \, \sigma(\mathcal{D}^{\infty}(z')) > 0 \big\} \Big] \geq c_2 \Vol\big( B(\delta') \big) ~.
\end{equation}

To conclude, let us first specify that (\ref{Fin-c2}) holds for any direction $u \in \mathbb{S}^d$ by isotropy of the model. Moreover, for any given $h_0 > 0$, the radius $r+h$ with $r > h_0$ can be written as $r'+h'$ where $r' > 0$ and $h' \in [h_0,h_0+A]$. Hence, we have proved that there exists $A$ large and $c_2=c_2(A,d,\lambda) > 0$ such that for any $0<\delta'<1$, $h_0 > 0$ and $r > 0$, there exists $h \in [h_0,h_0+A]$ such that (\ref{Fin-c2}) holds for any $u \in \mathbb{S}^d$. This is Lemma \ref{Lem:PositiveDensity}.

This last change on quantifiers, i.e. $r>h_0$ and $h \in [0,A]$ replaced with $r>0$ and $h \in [h_0,h_0+A]$, will allow us to take simultaneously $h$ in the good interval $[h_0,h_0+A]$ and also large enough. See Section \ref{sect:Conclusion}.
\end{proof}

\subsection{A stabilization result for subtrees of the RST}
\label{sect:Deterto1}

For $0<\delta'<\delta$ and $r,h>0$, recall that $z_2=(r+h;u)$ and recall the definition of $\Stab(r,h,\delta')$ in \eqref{def:Stab}. 
In this section, we prove a stabilization result: subtrees of the RST rooted at Poisson points in $B(z_2,\delta')$ do not depend on what happens inside $B(r)$ w.h.p. as $h\to\infty$.

\begin{lemma}
\label{Lem:LimitDeter}
For any $0<\delta'<\delta$,
\[
\lim_{h \to\infty} \, \sup_{r>0} \, \mathbb{P} ( \Stab(r,h,\delta') ) = 1 ~.
\]
\end{lemma}

\begin{proof}
Let us denote by $\chi$ the union of descendant sets $\mathcal{D}(z)$ with $z$ in $\mathcal{N} \cap B(z_2,\delta')$. So, for the configuration in $B(r)$ to alter $\chi$, it must exist a vertex $z'=(r';u')\in\chi$ whose $B^+(z',d(z',A(z')))$ overlaps $B(r)$, which means $d(z',A(z'))\geq r'-r$. Roughly speaking, this would imply the occurrence of a large ball empty of Poisson points with radius $r'-r \geq h-\delta'$ which is very unlikely as $h\to\infty$. Hence, $\Stab(r,h,\delta')^\complement$ is included in $\Romanbar{1}\cup\Romanbar{2}$ where, for a positive constant $M>0$,
\[
\Romanbar{1} := \big\{ \chi \not\subset \Cone(z_2 , M e^{-r-h}) \big\}
\]
and 
\[
\Romanbar{2} :=  \big\{ \exists z'=(r';u') \in \mathcal{N} \cap \Cone(z_2 , M e^{-r-h}) : \, d(z',A(z')) \ge r'-r \big\} ~ .
\]
We are going to prove that there exist $c,C>0$ (not depending on $r$, $h$ and $M$) such that
\begin{equation}
\label{IandII}
\forall r,h>0, \, \P(\Romanbar{1}) \leq \frac{C}{M^{d/2}} + e^{-cM} \; \mbox{ and } \; \lim_{h \to\infty} \, \sup_{r>0} \, \mathbb{P}(\Romanbar{2}) = 0 ~.
\end{equation}
The previous statement holding whatever the constant $M>0$, Lemma \ref{Lem:LimitDeter} then follows.\\

Let us deal with $\mathbb{P}(\Romanbar{2})$. By Mecke's formula and $\delta'<1$,
\begin{eqnarray*}
\mathbb{P}(\Romanbar{2}) & \leq & \E \big[ \#\big\{ z'=(r';u') \in \mathcal{N} \cap \Cone(z_2 , M e^{-r-h}) : \, d(z',A(z')) \ge r'-r \big\} \big] \\
& \leq & \sum_{n\geq h-1} \lambda \int_{V_n} \P \big( d(z',A(z')) \geq n \,|\, z' \in \mathcal{N} \big) \, dz'
\end{eqnarray*}
where $V_n := \Cone(z_2 , M e^{-r-h}) \cap C(r+n,r+n+1)$. On the one hand, for $z'\in V_n$,
\[
\P \big( d(z',A(z')) \geq n \,|\, z' \in \mathcal{N} \big) \leq \P \big( \mathcal{N} \cap B^+(z',n) = \emptyset \big) = e^{-\lambda \Vol(B^+(z',n))} ~.
\]
Here we need a lower bound of $B^+((r';\cdot) , \rho)$ and use the one obtained in \cite[eq. (A1)]{flammant-RST}: there exists $c=c(d)>0$ such that, for any radii $r',\rho\geq 1$,
\begin{equation}
\label{LowBoundB+}
\Vol \big( B^+((r';\cdot) , \rho) \big) \geq c e^{d (r' \wedge \rho)/2} ~.
\end{equation}
So we use (\ref{LowBoundB+}) to get 
\[\P(d(z',A(z'))\geq n \,|\, z' \in \mathcal{N}) \leq e^{-\lambda c e^{d n/2}}.\]

On the other hand, we upperbound the volume of $V_n$:
\begin{eqnarray*}
\Vol(V_n) & = & \int_{r+n}^{r+n+1} \sinh(\rho)^d \, d\rho \times \sigma \big( \{ u'\in\mathbb{S}^d : \, \widehat{u0u'} \le M e^{-r-h} \} \big) \\
& \leq & c_d e^{d(r+n)} \times M^d e^{-d(r+h)} \\
& = & c_d M^d e^{d(n-h)} ~.
\end{eqnarray*}
Putting together the previous upperbounds, we get:
\begin{equation}
\mathbb{P}(\Romanbar{2})  \leq  \lambda c_d M^d e^{-dh} \sum_{n\geq h-1} e^{dn -\lambda c e^{dn/2}},\label{etape3}
\end{equation}
which tends to $0$ as $h\to\infty$ uniformly on $r$.\\

Let us now consider the term $\mathbb{P}(\Romanbar{1})$. Let us denote by $\pi_z$ the path of $\RST$ starting from the Poisson point $z$ until the origin. We define $\text{Rad}(z_2)$ as the maximal angular deviation generated by a path $\pi_z$ starting from some $z\in B(z_2,\delta')$ when it goes through the sphere $S(r+h-\delta')$:
\[
\text{Rad}(z_2) := \max \big\{ \widehat{u0u'} : \, \exists z \in \mathcal{N} \cap B(z_2,\delta') \, \mbox{ s.t. $\pi_{z}$ intersects $S(r+h-\delta')$ at $(\cdot ; u')$} \big\} ~.
\]
The maximal angular deviation cannot be too large. Precisely, setting with $\text{Rad}:=\{\text{Rad}(z_2) \leq M e^{-(r+h-\delta')}\}$, Lemma 4.3 of \cite{flammant-RST} states that $\P(\Rad^\complement) \leq e^{-cM}$ for $c=c(d)>0$ and $M$ large enough. Besides,
\begin{eqnarray*}
\P ( \Romanbar{1} \cap \Rad ) & = & \P \big( \big\{ \exists z'=(r';\cdot) \in \mathcal{N} \cap B(z_2,\delta') : \, \MBD_{r'}^\infty(z') \geq M e^{-(r+h)} \big\} \cap \Rad \big) \\
& \leq & \E \Big( \ind_\Rad \sum_{z'\in\mathcal{N}\cap B(z_2,\delta')} \ind_{\{\MBD_{r'}^\infty(z')^d \geq M^d e^{-d(r+h)}\}} \Big) \\
& \leq & M^{-d} e^{d(r+h)} \E \Big[ \ind_\Rad \sum_{z'\in\mathcal{N}\cap B(z_2,\delta')} \MBD_{r'}^\infty(z')^d \Big].
\end{eqnarray*}
Now, in order to apply Lemma \ref{lem:MBD}, we have to turn the sum over elements of $\mathcal{N}\cap B(z_2,\delta')$ into a sum over elements of $\mathcal{L}_{r+h-\delta'}$. Given $z'=(r';\cdot)$, let us denote by $A_{r'}^{r+h-\delta'}(z')$ the intersection point between the path $\pi_{z'}$ of $\RST$ and $S(r+h-\delta')$. It is the `ancestor' at radius $r+h-\delta'<r'$ of $z'$ in $\RST$. A.s.
\begin{eqnarray*}
\lefteqn{\ind_\Rad \sum_{z'\in \mathcal{N} \cap B(z_2,\delta')} \MBD_{r'}^\infty(z')^{d}} \\
& & \leq \ind_\Rad \sum_{z'\in \mathcal{N} \cap B(z_2,\delta')} \MBD_{r+h-\delta'}^\infty(A_{r'}^{r+h-\delta'}(z'))^{d} \\
& & \leq \sum_{z'' \in B_{S(r+h-\delta')}(z_2,M e^{-(r+h-\delta')})\cap\RST} \# \big\{z'\in\mathcal{N}\cap B(z_2,\delta') : \ z' \in \mathcal{D}(z'') \big\} \MBD_{r+h-\delta'}^\infty(z'')^{d} .
\end{eqnarray*}
Thus, the Cauchy-Schwarz inequality gives:
\begin{eqnarray}
\label{2ndTermeCS}
\lefteqn{\E \Big[ \ind_\Rad \sum_{z'\in\mathcal{N}\cap B(z_2,\delta')} \MBD_{r'}^\infty(z')^d \Big]} \nonumber \\
& & \hspace*{1cm} \leq \E \Big[ \sum_{z'' \in \mathcal{L}_{r+h-\delta'}} \# \big\{z'\in\mathcal{N}\cap B(z_2,\delta') : \, z' \in \mathcal{D}(z'') \big\}^2 \Big]^{1/2} \nonumber \\
& & \hspace*{2cm} \times \E \Big[ \sum_{z'' \in B_{S(r+h-\delta')}(z_2,M e^{-(r+h-\delta')})\cap\RST} \MBD_{r+h-\delta'}^\infty(z'')^{2d} \Big]^{1/2} ~.
\end{eqnarray}
On the one hand, let us write
\begin{eqnarray*}
\lefteqn{\sum_{z'' \in \mathcal{L}_{r+h-\delta'}} \# \big\{z'\in\mathcal{N}\cap B(z_2,\delta') : \, z' \in \mathcal{D}(z'') \big\}^2} \\
& & \hspace*{2cm} \leq \Big( \sum_{z'' \in \mathcal{L}_{r+h-\delta'}} \# \big\{z'\in\mathcal{N}\cap B(z_2,\delta') : \, z' \in \mathcal{D}(z'') \big\} \Big)^2 \\
& & \hspace*{2cm} = \big( \# \mathcal{N} \cap B(z_2,\delta') \big)^2
\end{eqnarray*}
which means that the first term of the upper bound in (\ref{2ndTermeCS}) is bounded by $C_1 := \E[\#(\mathcal{N}\cap B(1))^2]^{1/2}$. Finally we apply Lemma \ref{lem:MBD} to the second term of the upper bound in (\ref{2ndTermeCS}): it is bounded by the square root of $C M^d e^{-2d(r+h-\delta')}$. Combining the previous bounds, we get 
\begin{equation}\label{etape4}
\P(\Romanbar{1} \cap \Rad)\leq C M^{-d/2}.
\end{equation}Gathering \eqref{etape3} and \eqref{etape4} proves Lemma \ref{Lem:LimitDeter}.
\end{proof}

\subsection{Conclusion}
\label{sect:Conclusion}

Let us prove Lemma \ref{Lem:unifF0}. We pay a special attention to dependencies between parameters. Let $A>0$ and $c_0=c_0(A,d,\lambda)>0$ given by Lemma \ref{Lem:PositiveDensity}. It is well known that
\begin{equation}
\label{Poisson>1}
\E \big[ \# \NN\cap B(z_2,\delta') \, \ind_{\# \NN \cap B(z_2,\delta') \geq 2} \big]
\end{equation}
which does not depend on $r,h$ by stationarity of the PPP, is negligible w.r.t. $\Vol(B(\delta'))$ as $\delta'\to 0$. Hence we choose $\delta'>0$ small enough and uniformly on $r,h>h_1=h_1(d)$ so that the expectation (\ref{Poisson>1}) and 
\[
\P \big( \exists z\in\mathcal{N}\cap B(z_2,\delta'),\ \mathcal{D}(z)\!\setminus\!\{z\} \not\subset B(r+h+\delta')^\complement \big)
\]
are both smaller than $\frac{c_0}{4} \Vol(B(\delta'))$. This second upper-bound and the constant $h_1=h_1(d)$ are proved in Lemma \ref{Lem:NegligibleDelta'}, whose proof is postponed in Section \ref{Sect:TechLemmas}. At this stage, parameters $r>0$ and $h>h_1$ are still free. Now we choose $h_0\geq h_1$ large enough so that for any $h\geq h_0$ and uniformly on $r>0$, the following holds by Lemma \ref{Lem:LimitDeter}:
\[
\P \big( \Stab(r,h,\delta')^\complement \big) \leq \frac{c_0}{4} \Vol(B(\delta')) ~.
\]
Finally, for any given radius $r>0$, we choose $h$ (possibly depending on $r$) in $[h_0,h_0+A]$ such that
\[
\E \Big[ \# \big\{ z \in \NN \cap B(z_2,\delta') : \, \sigma(\mathcal{D}^\infty(z)) > 0 \big\} \Big] \geq c_0 \Vol \big( B(\delta') \big)
\]
by Lemma \ref{Lem:PositiveDensity}.

For these fixed parameters $A,c_0,\delta',h_0,r$ and $h$,
\begin{eqnarray*}
c_0 \Vol(B(\delta')) & \leq & \mathbb{E} \big[ \# \big\{ z \in \NN \cap B(z_2,\delta') : \, \sigma(\mathcal{D}^\infty(z)) > 0 \big\} \big] \\
& \leq & \P \big( \exists ! z \in \NN \cap B(z_2,\delta') , \, \sigma(\mathcal{D}^\infty(z)) > 0 \big) \\
& & \hspace*{2cm}+ \E \big[ \# \{\NN\cap B(z_2,\delta')\} \, \ind_{\# \NN \cap B(z_2,\delta')\geq 2} \big] \\
& \leq & \P \big( \exists ! z \in \NN \cap B(z_2,\delta') , \, \sigma(\mathcal{D}^\infty(z)) > 0 \big) + \frac{c_0}{4} \Vol(B(\delta'))
\end{eqnarray*}
from which we get $\P(\exists ! z\in\NN\cap B(z_2,\delta'), \sigma(\mathcal{D}^\infty(z))>0) \geq \frac{3c_0}{4} \Vol(B(\delta'))$. Thus we conclude with
\begin{eqnarray*}
\P \big( F(r,h,\delta') \big) & \geq & \P \big( \exists ! z \in \NN \cap B(z_2,\delta') , \, \sigma(\mathcal{D}^\infty(z)) > 0 \big) \, - \, \P \big( \Stab(r,h,\delta')^\complement \big) \\
& & - \, \P \big( \exists z\in\mathcal{N}\cap B(z_2,\delta'),\ \mathcal{D}(z)\!\setminus\!\{z\} \not\subset B(r+h+\delta')^\complement \big) \\
& \geq & \frac{c_0}{4} \Vol(B(\delta')) ~.
\end{eqnarray*}
So $\eps := \frac{c_0}{4} \Vol(B(\delta'))$ depending on $A$, $d$, $\lambda$ and $\delta'$, is suitable, and Lemma \ref{Lem:unifF0} is proved.

\subsection{Technical lemmas}
\label{Sect:TechLemmas}

To complete the proof of Lemma \ref{Lem:unifF0}, it remains to state the two following technical lemmas.

\begin{lemma}
\label{Lem:Densite1/2}
There exists $c_1=c_1(d)>0$ such that for any $A$ large enough, $r>0$ and $0<\delta'<1$,
\begin{eqnarray*}
\lefteqn{\E \Big[ \# \big\{ z \in \mathcal{L}_r : \, z_{\downarrow} \in B(r+A-2\delta') \, \mbox{ and } \, \sigma(\mathcal{D}^{\infty}(z)) > 0 \big\} \Big]} \\
& & \hspace*{5cm}\geq c_1 \E \Big[ \# \big\{ z \in \mathcal{L}_r : \, \sigma(\mathcal{D}^{\infty}(z)) > 0 \big\} \Big]  ~.
\end{eqnarray*}
\end{lemma}

\begin{proof}
For the proof, we will consider first a portion of the sphere of radius $r$, and then extend the result to the whole $\mathcal{L}_r$ using the (Covering) Lemma \ref{Lem:covering}.\\

\noindent \textbf{Step 1:}
Given $z_0\in S(r)$, let us consider the event 
\[
L(A,r) := \big\{ \exists z \in B_{S(r)}(z_0,e^{-r}) \cap \mathcal{L}_r : \, z_{\downarrow} \notin B(r+A-2\delta') \big\} ~.
\]
For short, let us set $X_{0,r} := \# \big\{ z \in B_{S(r)}(z_0,e^{-r}) \cap \mathcal{L}_r : \, \sigma(\mathcal{D}^{\infty}(z)) > 0 \big\}$. 
We will prove that 
\begin{equation}\label{etape5}
\E \big[ X_{0,r} \ind_{L(A,r)^\complement} \big] \geq \frac{1}{2}\E[X_{0,r}] ~ .
\end{equation}
With
\[
X_{0,r} \ind_{L(A,r)^\complement} = \# \big\{ z \in B_{S(r)}(z_0,e^{-r}) \cap \mathcal{L}_r : \, z_{\downarrow} \in B(r+A-2\delta') \, \mbox{ and } \, \sigma(\mathcal{D}^{\infty}(z)) > 0 \big\}~ .
\]
Lemma \ref{Lem:Densite1/2} then immediatly follows from the above inequality with $c_1:=1/(2K)$ using the (Covering) Lemma \ref{Lem:covering}, where the covering constant $K=K(d)$ is given in that lemma.\\

The proof of \eqref{etape5} relies on the two following inequalities \eqref{UnifL(A,r)} and \eqref{MinorEsigma} which will be proved in a second time.\\
The first inequality states that the probability of $L(A,r)$ tends to $0$ when $A\to\infty$, uniformly on $r,\delta'$:
\begin{equation}
\label{UnifL(A,r)}
\lim_{A\to\infty} \sup_{r>0, \ \delta'<1} \P( L(A,r) ) = 0 ~.
\end{equation}
The second inequality ensures that there exists $c=c(d)>0$ such that for any $r>0$,
\begin{equation}
\label{MinorEsigma}
\E \big[ X_{0,r}\big] \geq c ~.
\end{equation}
Then, the Cauchy-Schwarz inequality gives
\[
\E \big[ X_{0,r} \ind_{L(A,r)} \big] \leq C^{1/2} \P( L(A,r) )^{1/2}
\]
where $C=C(d)>0$ upperbounds the expectation of $\#\big(B_{S(r)}(z_0,e^{-r}) \cap \mathcal{L}_r\big)^2$ thanks to Lemma \ref{lem:MomLr}. Thus, combining \eqref{UnifL(A,r)} and \eqref{MinorEsigma}, we can choose $A$ large enough uniformly on $r\geq 2$ and $\delta'<1$ such that 
\[
\E \big[ X_{0,r} \ind_{L(A,r)} \big] \leq \frac{1}{2} \E \big[ X_{0,r} \big] ~.
\]This proves \eqref{etape5}.\\

\noindent \textbf{Step 2: proof of \eqref{UnifL(A,r)}.} The proof  is very close to that of Lemma \ref{Lem:LimitDeter} with fewer technical difficulties. By analogy, we write $L(A,r)\subset \Romanbar{1}\cup\Romanbar{2}$ where
\[
\Romanbar{1} := \big\{ \exists z \in B_{S(r)}(z_0,e^{-r}) \cap \mathcal{L}_r : \, \mathcal{D}(z) \not\subset \Cone(z_0 , A e^{-r}) \big\}
\]
and 
\[
\Romanbar{2} :=  \big\{ \exists z \in B_{S(r)}(z_0,e^{-r}) \cap \mathcal{L}_r : \, z_{\downarrow} \in \Cone(z_0 , A e^{-r}) \cap B(r+A-2\delta')^\complement \big\} ~.
\]
We upperbound $\P(\Romanbar{1})$ more easily than in the proof of Lemma \ref{Lem:LimitDeter} since the points $z$ concerned by the event $\Romanbar{1}$ are already in $\mathcal{L}_r$:
\begin{eqnarray*}
\P( \Romanbar{1} ) & \leq & \P \big( \exists z \in B_{S(r)}(z_0,e^{-r}) \cap \mathcal{L}_r : \, \MBD_{r}^{\infty}(z) \geq A e^{-r} \big) \\
& \leq & A^{-2d} e^{2dr} \, \E \Big[ \sum_{z \in B_{S(r)}(z_0,e^{-r}) \cap \mathcal{L}_r} \big( \MBD_{r}^{\infty}(z) \big)^{2d} \Big] \, \leq \, C A^{-2d}
\end{eqnarray*}
(by Lemma \ref{lem:MBD}) which tends to $0$ uniformly on $r,\delta'$. The same holds for $\P(\Romanbar{2})$ as in the proof of Lemma \ref{Lem:LimitDeter}. For this reason, we omit the details.\\

\noindent \textbf{Step 3: proof of \eqref{MinorEsigma}.} This inequality is a consequence of isotropy of the model, (Covering) Lemma \ref{Lem:covering} and Step 1 in the proof of Lemma \ref{Lem:PositiveDensity}:
\begin{eqnarray*}
\lefteqn{\E \Big[ \# \big\{ z \in B_{S(r)}(z_0,e^{-r}) \cap \mathcal{L}_r : \, \sigma(\mathcal{D}^{\infty}(z)) > 0 \big\} \Big]} \\
& &\hspace*{1cm}= \frac{1}{N(r)} \sum_{i=1}^{N(r)} \E \Big[ \# \big\{ z \in B_{S(r)}(z_i,e^{-r}) \cap \mathcal{L}_r : \, \sigma(\mathcal{D}^{\infty}(z)) > 0 \big\} \Big] \\
& &\hspace*{1cm}\geq \frac{1}{N(r)} \E \Big[ \# \big\{ z \in \mathcal{L}_r : \, \sigma(\mathcal{D}^{\infty}(z)) > 0 \big\} \Big] \\
& &\hspace*{1cm}\geq c \, C^{-1}
\end{eqnarray*}
with $N(r)\leq C e^{dr}$. This concludes the proof of the Lemma.
\end{proof}

\begin{lemma}
\label{Lem:NegligibleDelta'}
There exists $h_1=h_1(d)>0$ such that the following limit is uniform on $r>0$ and $h>h_1$:
\[
\lim_{\delta' \to 0} \frac{1}{\Vol(B(\delta'))} \P \big( \exists z\in\mathcal{N}\cap B(z_2,\delta') , \, \mathcal{D}(z)\!\setminus\!\{z\} \not\subset B(r+h+\delta')^\complement \big) = 0 ~.
\]
\end{lemma}

\begin{proof}
The Mecke's formula allows to write:
\begin{eqnarray}
\label{NegligibleDelta'}
\lefteqn{\P \big( \exists z\in\mathcal{N}\cap B(z_2,\delta') , \, \mathcal{D}(z)\!\setminus\!\{z\} \not\subset B(r+h+\delta')^\complement \big)} \nonumber \\
& & \hspace*{1cm} \leq \P \big( \exists z\in\mathcal{N}\cap B(z_2,\delta') , \, \exists z'\in \mathcal{N} \cap B(r+h+\delta') \, \mbox{ s.t. } A(z')=z \big) \nonumber\\
& & \hspace*{1cm} \leq \E \big[ \#\big\{ z\in\mathcal{N}\cap B(z_2,\delta') : \, \exists z'\in \mathcal{N} \cap B(r+h+\delta') \, \mbox{ s.t. } A(z')=z \big\} \big] \nonumber \\
& & \hspace*{1cm} = \lambda \int_{B(z_2,\delta')} \P \big( \exists z'\in \mathcal{N} \cap B(r+h+\delta') \, \mbox{ s.t. } A(z')=z \,|\, z\in\mathcal{N} \big) \, dz ~.
\end{eqnarray}
For $z$ in $B(z_2,\delta')$,
\begin{eqnarray*}
\lefteqn{\P \big( \exists z'\in \mathcal{N} \cap B(r+h+\delta') \, \mbox{ s.t. } A(z')=z \,|\, z\in\mathcal{N} \big)} \\
& & \hspace*{1cm}  \leq \sum_{n\geq 0} \E \big[ \#\big\{ z' \in \mathcal{N} \cap V_{\delta',n} : \, A(z')=z \big\} \,|\, z\in\mathcal{N} \big]
\end{eqnarray*}
where we set
\[
V_{\delta',n} := C(r+h-\delta',r+h+\delta') \cap \Cone(u,n e^{-r-h},(n+1) e^{-r-h})
\]
and $\Cone(u,ne^{-r-h},(n+1)e^{-r-h})$ is the set of directions $u'\in \mathbb{S}^d$ such that $n e^{-r-h}\leq \widehat{u0u'}\leq (n+1) e^{-r-h}$. A second use of the Mecke's formula gives:
\[
\E \big[ \#\big\{ z' \in \mathcal{N} \cap V_{\delta',n} : \, A(z')=z \big\} \,|\, z\in\mathcal{N} \big] = \lambda \int_{V_{\delta',n}} \P \big( A(z')=z \,|\, z,z' \in\mathcal{N} \big) \, dz' ~.
\]
Given $z'\in V_{\delta',n}$, we have
\[
\P \big( A(z')=z \,|\, z,z' \in\mathcal{N} \big) = \P \big( B^+(z',d(z,z')) \cap \mathcal{N} = \emptyset \big) = e^{-\lambda \Vol(B^+(z',d(z,z')))} ~.
\]
Moreover
\begin{eqnarray*}
\Vol(B^+(z',d(z,z'))) & \geq & c e^{\frac{d}{2} d(0,z')\wedge d(z,z')} \; \; \mbox{ by (\ref{LowBoundB+})} \\
& \geq & c e^{\frac{d}{4} d(z,z')} \; \; \mbox{ since $d(z,z') \leq 2 d(0,z')$} \\
& \geq & c e^{\frac{d}{4} c' n}
\end{eqnarray*}
because $z'\in V_{\delta',n}$ and $r+h$ is larger than some $h_1=h_1(d)>0$.

Let us now compute the volume of $V_{\delta',n}$:
\begin{equation}
\label{VolumeVdelta'n}
\Vol(V_{\delta',n}) = \int_{r+h-\delta'}^{r+h+\delta'} \sinh(\rho)^d \, d\rho \times \sigma \big( \{ u'\in\mathbb{S}^d : \, n e^{-r-h} \leq \widehat{u0u'} \le (n+1) e^{-r-h} \} \big) ~.
\end{equation}
The first term of the r.h.s. of (\ref{VolumeVdelta'n}) is bounded by $c_1 \delta' e^{d(r+h)}$ while the second one is bounded by  $c_2 n^d e^{-d(r+h)}$ using \eqref{InegVol}. The previous constants $c_i$, $i\in\{1,2\}$, are positive and only depend on $d$. Hence, the volume of $V_{\delta',n}$ is smaller than $c_1c_2 n^d \delta'$. Combining what precedees, we finally get for any $r>0$ and $h>h_1$
\[
\P \big( \exists z'\in \mathcal{N} \cap B(r+h+\delta') \, \mbox{ s.t. } A(z')=z \,|\, z\in\mathcal{N} \big) \leq \sum_{n\geq 0} \lambda e^{-\lambda c e^{\frac{d}{4} c' n}} c_1c_2 n^d \delta'
\]
whose upperbound can be expressed as $C \delta'$ with $C=C(\lambda,d,h_1)$ is a positive, finite constant. It then remains to plug this bound in (\ref{NegligibleDelta'}) to conclude.
\end{proof}

\section{Proof of Proposition \ref{Propo:DensitySigma>0}}\label{Section:Densite>0}

Let us first recall an upperbound for the number of elements of $\mathcal{L}_{r}$ in a cap $B_{S(r)}(\cdot,e^{-r})$.

\begin{lemma}[Lemma 4.4 of \cite{flammant-RST}]
\label{lem:MomLr}
For any $p \ge 1$, there exists a constant $C=C(d,p)>0$ such that, for any $r \ge 0$ and any direction $z \in S(r)$,
\[
\mathbb{E} \Big[ \# \big( \RST \cap B_{S(r)}(z,e^{-r}) \big)^p \Big] \leq C ~.
\]
\end{lemma}

Thanks to the Covering Lemma (Lemma \ref{Lem:covering}), it is now easy to prove that $\mathcal{L}_r = \RST \cap S(r)$ admits in expectation about $e^{dr}$ elements:
\begin{equation}\label{utile:prop3.6}
\E \big[ \# \mathcal{L}_r \big] \leq \sum_{i=1}^{N(r)} \E \big[ \# ( B_{S(r)}(z_i,e^{-r}) \cap \mathcal{L}_r) \big] \leq C e^{dr} ~.
\end{equation}
Finally, combining the previous inequality and \eqref{densite>0-Lr} proved in Step 1 of Lemma \ref{Lem:PositiveDensity}, we get Proposition \ref{Propo:DensitySigma>0}.

\appendix 

\section{Arcs of the RST as a subset of $\mathbb{H}^{d+1}$}\label{Appendix:arcs}

Given $z_1,z_2\in \mathbb{H}^{d+1}$, the arc $|\![z_1,z_2]\!|$ is precisely defined in Section 2.2 of \cite{flammant-RST} but for convenience we recall the main lines. Let us write $z_i=(r_i;u_i)$ with polar cooordinates, $i\in \{1,2\}$. Whenever $u_1$ and $u_2$ are not antipodal (and it will be a.s. the case when $z_2$ is the ancestor of $z_1$), we consider the unique geodesic $\gamma_{u_1,u_2}:[0,1] \to \mathbb{S}^{d}$ on the sphere connecting $u_1$ to $u_2$. Then, the arc $|\![z_1,z_2]\!|$ is the path
\[
t \in [0,1] \, \mapsto \, \big((1-t)r_1+tr_2 ; \gamma_{u_1,u_2}(\phi_{z_1,z_2}(t)) \big) \in \mathbb{H}^{d+1}
\]
where $\phi_{z_1,z_2}:[0,1] \to [0,1]$ is defined as
\[
\phi_{z_1,z_2}(t) := \frac{1}{\widehat{u_10u_2}} \arccos \Big(\frac{(1-t)\sinh(r_1)+t\cos(\widehat{u_10u_2})\sinh(r_2)}{\sinh((1-t)r_1+tr_2)} \Big) ~.
\]
This construction of the arc $|\![z_1,z_2]\!|$ ensures that the distance to the origin (as well as the distance to $z_1$) are monotonous along the path $|\![z_1,z_2]\!|$.\\

By Property \ref{Prop:FiniteDegree}, we know that the geodesics $[z,A(z)]$ and $[z',A(z')]$, for $z,z'\in \NN$, can overlap only at their extremities. This is not the case any more when we use the arcs $|\![z,A(z)]\!|$ and $|\![z',A(z')]\!|$. For any $r>0$, it may exist some points $z\in\mathcal{L}_r$ belonging to several arcs, say $|\![z_1,A(z_1)]\!|,\ldots,|\![z_k,A(z_k)]\!|$. Hence, such a point $z$ will be counted with multiplicity $k$ in $\mathcal{L}_r$. Also, to identify without ambiguity this point $z$, we should formally represent it as a couple made up with its location in $\mathbb{H}^{d+1}$ and one of the arcs generating it, say $|\![z_i,A(z_i)]\!|$. In this case the vertex $z_\downarrow\in\NN$ is defined as $z_\downarrow := z_i$. In this article, we will commit the following abuse of notations: we will count elements of $\mathcal{L}_r$ with multiplicity without specifying the arcs distinguishing them.


\end{document}